\documentclass[10pt]{amsart}
\usepackage{amsmath, amscd, amsthm} 
\usepackage{amssymb, amsfonts} 
\usepackage{enumerate}
\usepackage{bbm}
\usepackage{color}
\usepackage[all]{xy} 
\usepackage[margin=1.1in]{geometry}
\usepackage{palatino}
\usepackage{mathrsfs}
\usepackage{booktabs}
\usepackage{aliascnt}
\usepackage{mathrsfs}
\usepackage[svgnames]{xcolor} 
\usepackage{hyperref}
\definecolor{dark-red}{rgb}{0.4,0.15,0.15}

\hypersetup{
    colorlinks, linkcolor=dark-red,
    citecolor=DarkBlue, urlcolor=MediumBlue
}
\usepackage{mathrsfs}
\usepackage[
textwidth=3cm,
textsize=small,
colorinlistoftodos]
{todonotes}

\usepackage{tikz}

\renewcommand{\AA}{\mathbb{A}}

\newcommand{\CC}{\mathbb{C}}
\newcommand{\GG}{\mathbb{G}}

\newcommand{\ZZ}{\mathbb{Z}}
\newcommand{\Sch}{\mathrm{Sch}}
\newcommand{\Sm}{\mathrm{Sm}}

\newcommand{\pr}{\mathrm{pr}}
\newcommand{\II}{\mathscr{I}}
\newcommand{\Bl}{\mathrm{Bl}}
\newcommand{\Spec}{\operatorname{Spec}}
\newcommand{\supp}{\operatorname{supp}}

\newcommand{\ML}{\operatorname{ML}}
\newcommand{\ddiv}{\mathrm{div}}

\newcommand{\R}{\mathbb{R}}
\newcommand{\RR}{\R}

\newcommand{\iso}{\cong}

\newcommand{\id}{\mathrm{id}}

\renewcommand{\setminus}{\smallsetminus}

\usepackage[T1]{fontenc}

\numberwithin{equation}{section} 
\theoremstyle{plain}
\theoremstyle{definition}
 
\newaliascnt{theorem}{equation}  
\newtheorem{theorem}[theorem]{Theorem}  
\aliascntresetthe{theorem}

\newaliascnt{prop}{equation}  
\newtheorem{prop}[prop]{Proposition}
\aliascntresetthe{prop}

\newaliascnt{lemma}{equation}  
\newtheorem{lemma}[lemma]{Lemma}
\aliascntresetthe{lemma}

\newaliascnt{corollary}{equation}  
\newtheorem{corollary}[corollary]{Corollary}
\aliascntresetthe{corollary}

\newaliascnt{claim}{equation}  

\aliascntresetthe{claim}

\newaliascnt{conjecture}{equation}  

\aliascntresetthe{conjecture}

\newaliascnt{question}{equation}  
\newtheorem{question}[question]{Question}
\aliascntresetthe{question}

\newaliascnt{defn}{equation}  
\newtheorem{defn}[defn]{Definition}
\aliascntresetthe{defn}

\newaliascnt{example}{equation}  
\newtheorem{example}[example]{Example}
\aliascntresetthe{example}

\theoremstyle{remark}

\newaliascnt{remark}{equation}  
\newtheorem{remark}[remark]{Remark}
\aliascntresetthe{remark}

\newaliascnt{convention}{equation}  

\aliascntresetthe{convention}

\newcommand{\aref}[1]{\autoref{#1}}

\newcommand{\PP}{\mathbb{P}}

\newcommand{\CH}{\operatorname{CH}}
\newcommand{\SSing}{\operatorname{Sing}}

\begin{document}
\title{$\AA^{1}$-contractibility of affine modifications}
\author{Adrien Dubouloz}
\address{IMB UMR5584, CNRS, Universit{\'e} de Bourgogne Franche-Comt{\'e}, Dijon, France}
\email{Adrien.Dubouloz@u-bourgogne.fr}
\author{Sabrina Pauli}
\address{Department of Mathematics, University of Oslo, Norway}
\email{sabrinp@student.matnat.uio.no}
\author{Paul Arne {\O}stv{\ae}r}
\address{Department of Mathematics, University of Oslo, Norway}
\email{paularne@math.uio.no}
\begin{abstract}
We introduce Koras-Russell fiber bundles over algebraically closed fields of characteristic zero.
After a single suspension, 
this exhibits an infinite family of smooth affine $\AA^{1}$-contractible $3$-folds.
Moreover, 
we give examples of stably $\AA^{1}$-contractible smooth affine $4$-folds containing a Brieskorn-Pham surface, 
and a family of smooth affine $3$-folds with a higher dimensional $\mathbb{A}^{1}$-contractible total space.
\end{abstract}
\maketitle

\section{Introduction}
\label{sec:intro}
Suppose that $\mathcal{X}$ is some $d$-dimensional smooth affine $\AA^{1}$-contractible scheme of finite type over an algebraically closed field $F$ of characteristic zero.
Noether normalization shows there exists a branched covering and $\AA^{1}$-weak equivalence $\mathcal{X}\to\AA^{d}$, 
resembling a covering space in the sense of topology. 

When $d=1$, 
then $\mathcal{X}$ is isomorphic to $\AA^{1}$ \cite[Claim 5.7]{MR2335246}. 
Any $\AA^{1}$-chain connected smooth affine surface in the sense of Asok-Morel \cite{MR2803793} is $\AA^{1}$-ruled,
i.e., 
it contains a cylinderlike open subscheme of the form $\mathcal{X}'\times\AA^{1}$,
and hence of logarithmic Kodaira dimension $-\infty$ \cite{MR3351185}.
Thus, 
when $d=2$, 
deep classification results in birational geometry due to Miyanishi-Sugie and Fujita \cite{MR564667} imply that $\mathcal{X}$ is isomorphic to $\AA^{2}$ 
(see \aref{subsection:a1contractiblesmoothaffinesurfaces}).
Thus $\AA^{d}$ is the unique up to isomorphism smooth affine $\AA^{1}$-contractible and $\AA^{1}$-chain connected scheme of dimension $d\leq 2$. 
Such a characterization fails for $\AA^{3}$, 
as recently witnessed by Koras-Russell $3$-folds of the first kind \cite{MRDF}, \cite{MR3549169}.
Initially used in the monumental work of showing every $\GG_{m}$-action on the affine $3$-space is linearizable \cite{MR1464577},
these $3$-folds are of interest in their own right as potential counterexamples to the long-standing Zariski cancellation problem \cite{MR1423629}, \cite{MR1734345}.
Moreover,
they admit hyperbolic $\GG_{m}$-actions with a unique fixed point, 
and can be distinguished from $\AA^{3}$ by the Derksen and Makar-Limanov invariants recording locally nilpotent derivations on coordinate rings \cite{MR2259515}.

For integers $n,\alpha_{i}\geq 2$, 
where $\alpha_{1}$ and $\alpha_{2}$ are coprime, 
$a\in F^{\times}$, 
recall that Koras-Russell $3$-folds of the first kind are defined in $\AA^{4}$ by the equation:
\begin{equation}
\label{equation:firstkind}
\mathcal{X}(n,\alpha_{i},a)
:=
\{
x^{n}z=
y^{\alpha_{1}}+t^{\alpha_{2}}+ax
\}.
\end{equation} 
In \cite{MR3549169} it is shown that Koras-Russell $3$-folds of the first kind are stably $\AA^{1}$-contractible,
and consequently $\AA^{1}$-contractible after a finite suspension by the projective line $\PP^{1}$ pointed at infinity:
\begin{equation}
\label{equation:stableA1KR}
\mathcal{X}(n,\alpha_{i},a)\wedge \PP^{1}\wedge \cdots \wedge \PP^{1}\sim_{\AA^{1}}\ast.
\end{equation}
Similarly by \cite{MR3549169}, 
for $m\geq 1$, $n,\alpha_{i}\geq 2$, 
where $\alpha_{1}$ and $n\alpha_{2}$ are coprime,
$a\in F^{\times}$,
the same $\AA^{1}$-contractibility result hold for Koras-Russell $3$-folds of the second kind, 
defined in $\AA^{4}$ by the equation:
\begin{equation}
\label{equation:secondkind}
\mathcal{X}(m,n,\alpha_{i},a)
:=
\{
(x^{n}+y^{\alpha_{1}})^{m}z=
t^{\alpha_{2}}+ax
\}.
\end{equation} 

As shown in \cite{MRDF}, 
$\mathcal{X}(n,\alpha_{i},a)$ is in fact $\AA^{1}$-contractible. 
This makes \eqref{equation:firstkind} the first family of examples of smooth affine $\AA^{1}$-contractible schemes that are not isomorphic to affine spaces.
One cannot make such a distinction topologically because $\mathcal{X}(n,\alpha_{i},a)$ is diffeomorphic to $\RR^{6}$ when equipped with the Euclidean topology;
this was shown independently by Dimca and Ramanujam \cite[\S3]{MR1734345}.
From the viewpoint of motivic homotopy theory, 
the remarkable properties of a Koras-Russell $3$-fold of the first kind makes it an algebro-geometric analogue of the Whitehead manifold.
The latter is an open contractible $3$-manifold non-homeomorphic to $\RR^{3}$, 
and the first counterexample to the assertion that every contractible manifold is homeomorphic to a ball, 
as stated in a purported proof of the Poincar\'e conjecture \cite{zbMATH03015497}, \cite{zbMATH03019987}.
It is an open question whether Koras-Russell $3$-folds of the second kind \eqref{equation:secondkind} are $\AA^{1}$-contractible.

In this paper we extend the results in \cite{MRDF} and \cite{MR3549169} to classes of examples such as deformed Koras-Russell $3$-folds of the first kind
\begin{equation}
\label{equation:firstkindpolynomial}
\mathcal{X}(n,\alpha_{i},p)
:=
\{
x^{n}z=
y^{\alpha_{1}}+t^{\alpha_{2}}+xp(x,y,t)
\},
\end{equation} 
where $n,\alpha_{i}\geq 2$ are integers, $\alpha_{1}$ and $\alpha_{2}$ are coprime, and $p(x,y,t)\in F[x,y,t]$ satisfies $p(0,0,0)\in F^{\ast}$.
The special case $p(x,y,t)=a\in F^{\ast}$ recovers \eqref{equation:firstkind}.
By \cite{MRDF} it is known that $\mathcal{X}(n,\alpha_{i},p)$ is $\AA^{1}$-contractible when $p(x,y,t)=q(x)\in F[x]$ and $q(0)\in F^{\ast}$.
Note that $\mathcal{X}(n,\alpha_{i},p)$ is smooth according to the Jacobian criterion since we assume $p(0,0,0)\in F^{\ast}$. 
Moreover, 
the unique singular fiber of the projection map 
\begin{equation}
\label{equation:firstkindpolynomialprojection}
\pr_{x}:
\mathcal{X}(n,\alpha_{i},p)
\to
\AA^{1}_{x}
\end{equation} 
is a cylinder on the cuspidal curve $\Gamma_{\alpha_{1},\alpha_{2}}:=\{y^{\alpha_{1}}+t^{\alpha_{2}}=0\}\subset\AA^{2}$, which is $\AA^{1}$-contractible:
\begin{equation}
\label{equation:firstkindpolynomialinverse}
\pr_{x}^{-1}(0)
=
\Gamma_{\alpha_{1},\alpha_{2}}\times \AA^{1}_{z}
\sim_{\AA^{1}}
\ast.
\end{equation} 
Here \eqref{equation:firstkindpolynomialprojection} is a flat $\AA^{2}$-fibration,
i.e., 
a flat surjection with general fibers isomorphic to $\AA^{2}$, 
restricting to a trivial $\AA^{2}$-bundle over $\AA^{1}_{x}\setminus\{0\}$ and $\mathcal{X}(n,\alpha_{i},p)$ is factorial 
(see e.g., \cite[Lemma 3.1]{MR3223879}).
The $\AA^{1}$-homotopy theory of deformed Koras-Russell $3$-folds of the first kind \eqref{equation:firstkindpolynomial} is essentially governed by \eqref{equation:firstkindpolynomialprojection} 
and \eqref{equation:firstkindpolynomialinverse}.
More generally,
we make the following definition.
\begin{defn}
\label{def:KRBundle}
Suppose $s(x)\in F[x]$ has positive degree and let $R(x,y,t)\in F[x,y,t]$.
Define the closed subscheme $\mathcal{X}(s,R)$ of $\AA^{1}_{x}\times\AA^{3}=\Spec(F[x][y,z,t])$ by the equation: 
\[
\{s(x)z=R(x,y,t)\}.
\]
We say that the projection map  
\begin{equation}
\label{equation:KRfiberbundleprojection}
\rho
:=
\pr_{x}:
\mathcal{X}(s,R)
\to
\AA^{1}_{x}
\end{equation} 
defines a {\it Koras-Russell fiber bundle }if
\begin{itemize}
\item[(a)] $\mathcal{X}(s,R)$ is a smooth scheme, and
\item[(b)] For every zero $x_{0}$ of $s(x)$, the zero locus in $\AA^2=\Spec(F[y,t])$ of the polynomial $R(x_0,y,t)$ is an integral rational plane curve with a unique place at infinity 
and at most unibranch singularities.
\end{itemize}
\end{defn}
\begin{remark}
\label{remark:KRaffinespace}
\aref{theorem:Kalimancriterion} and \aref{theorem:miyanishi's characterization} show that a Koras-Russell fiber bundle is isomorphic to $\AA^{3}$ if and only if for every zero $x_0$ of $s(x)$ 
the curve $\{R(x_0,y,t)=0\}$ is isomorphic to $\AA^{1}$.
For the theory of unibranch singularities, we refer to \cite{zbMATH03950697}.
\end{remark}

As examples of Koras-Russell fiber bundles we consider smooth affine $3$-folds with degenerate fibers.
For $1\leq i\leq m$, 
choose distinct linear forms $l_{i}(x)=(x-x_{i})\in F[x]$, 
$n_{i},\alpha_{i},\beta_{i}\geq 2$,
where $\alpha_{i}$ and $\beta_{i}$ are coprime, 
and $a\in F^{\times}$.
We define $\mathcal{X}_{m}(n_{i},\alpha_{i},\beta_{i},a)$ or simply $\mathcal{X}_{m}$ by the equation in $\AA^{4}$:
\begin{equation}
\label{equation:mdegeneratefibers}
\mathcal{X}_{m}
=
\mathcal{X}_{m}(n_{i},\alpha_{j},\beta_{j},a)
:=
\{\left(\prod_{i=1}^{m} l_{i}(x)^{n_i}\right)z
=
\sum_{i=1}^{m}\left(\left(\prod_{j\neq i}l_{j}(x)\right)(y^{\alpha_{i}}+t^{\beta_{i}})\right)+a\prod_{i=1}^{m}l_{i}(x)\}.
\end{equation}
In the case of two degenerate fibers, 
\eqref{equation:mdegeneratefibers} takes the form:
\begin{equation}
\label{equation:twodegeneratefibers}
\mathcal{X}_{2}
=
\{(x-x_{1})^{n_{1}}(x-x_{2})^{n_{2}}z
=
(x-x_{1})(y^{\alpha_{2}}+t^{\beta_{2}})
+
(x-x_{2})(y^{\alpha_{1}}+t^{\beta_{1}})
+
a(x-x_{1})(x-x_{2})\}.
\end{equation}
For all $m$ the Makar-Limanov invariant of $\mathcal{X}_{m}(n_{i},\alpha_{j},\beta_{j},a)$ equals $F[x]$ (see \aref{example:MLandDerksenInvariant}); 
hence we can distinguish it from $\AA^{3}$.
Moreover, 
the projection map
\begin{equation}
\label{equation:degeneratefirstkindpolynomialprojection}
\pr_{x}:
\mathcal{X}_{m}(n_{i},\alpha_{j},\beta_{j},a)
\to
\AA^{1}_{x}
\end{equation} 
defines a trivial $\AA^{2}$-bundle over the punctured affine line $\AA^{1}_{x}\setminus\{x_{1},\dots,x_{m}\}$.
Its fiber over the closed point $x_{i}\in\AA^{1}_{x}$ is isomorphic to the cylinder on the cuspidal curve $\Gamma_{\alpha_{i},\beta_{i}}:\{y^{\alpha_{i}}+t^{\beta_{i}}=0\}$.
By counting closed fibers non-isomorphic to $\AA^{2}$ we show $\mathcal{X}_{m}$ and $\mathcal{X}_{m'}$ are non-isomorphic when $m\neq m'$ (see \aref{lemma:mnonisomorphic}).
Thus the next result furnish an infinite family of stably $\AA^{1}$-contractible smooth affine schemes.

\begin{theorem}
\label{thm:KRBundle-Stable-A1Cont}
Suppose $\rho:\mathcal{X}(s,R)\to\AA^{1}_{x}$ is a Koras-Russell fiber bundle with basepoint the origin.
There are $\AA^{1}$-contractible suspensions: 
\[
\mathcal{X}(s,R)\wedge S^{1}\sim_{\AA^{1}}
\mathcal{X}(s,R)\wedge\PP^{1}\sim_{\AA^{1}}
\ast.
\] 
\end{theorem}

\begin{remark}
In general, 
$\AA^{1}$-weak equivalences do not desuspend, 
e.g., 
by homotopy purity we have $(\GG_{m}\vee\GG_{m})\wedge\PP^{1}\sim_{\AA^{1}} (\PP^{1}\setminus\{0,1,\infty\} = \Spec(F[t^{\pm 1},(1-t)^{-1}]))\wedge\PP^{1}$, 
but there is no $\AA^{1}$-weak equivalence between $\GG_{m}\vee\GG_{m}$ and $\PP^{1}\setminus\{0,1,\infty\}$ by \cite{zbMATH06617521}.
\aref{thm:KRBundle-Stable-A1Cont} does not allows us to conclude that $\mathcal{X}(s,R)$ is $\AA^{1}$-contractible.
\end{remark}

A systematic study of $\AA^{1}$-contractible varieties presents itself by noting that deformed Koras-Russell $3$-folds are examples of affine modifications \cite{zbMATH01286199}. 
It turns out that forming such "affine blow-ups" are intertwined with questions about $\AA^{1}$-contractibility.
This provides a useful technique for constructing families of $\AA^{1}$-contractible smooth affine varieties.
By way of example, 
we consider affine modifications of deformed Koras-Russell $3$-folds of the first kind \eqref{equation:firstkindpolynomial}.
This yields iterated Koras-Russell $3$-folds of the first kind defined in $\AA^{4}$ by 
\begin{equation}
\label{equation:iteratedfirstkindpolynomial}
\mathcal{Y}(m,n_{i},\alpha_{i},p)
:=
\{x^{n_{1}}z=
(x^{n_{2}}y+z^{m})^{\alpha_{1}}+t^{\alpha_{2}}+xp(x,x^{n_2}y+z^m,t)\},
\end{equation} 
where $m,n_{i},\alpha_{i}\geq 2$, $m\alpha_{1}$ and $\alpha_{2}$ are coprime, and $p(x,w,t)\in F[x,w,t]$ satisfies $p(0,0,0)\in F^{\ast}$. 
The projection map $\pr_{x}:\mathcal{Y}(m,n_{i},\alpha_{i},p)\to\AA^{1}_{x}$ has a unique singular fiber $\AA^{1}\times\Gamma_{m\alpha_{1},\alpha_{2}}$ at the origin and restricts to 
a trivial $\AA^{2}$-bundle over $\AA^{1}_{x}\setminus\{0\}$.
By \aref{theorem:Kalimancriterion} it follows that $\mathcal{Y}(m,n_{i},\alpha_{i},p)$ in \eqref{equation:iteratedfirstkindpolynomial} is not isomorphic to $\AA^{3}$.

\begin{theorem}
\label{thm:iteratedKR-A1Cont}
For every iterated Koras-Russell $3$-fold of the first kind \eqref{equation:iteratedfirstkindpolynomial} there are $\AA^{1}$-contractible suspensions: 
\[
\mathcal{Y}(m,n_{i},\alpha_{i},p)\wedge S^{1}\sim_{\AA^{1}}
\mathcal{Y}(m,n_{i},\alpha_{i},p)\wedge\PP^{1}\sim_{\AA^{1}}
\ast.
\]
\end{theorem}

Inspired by these $3$-folds and the classical Brieskorn-Pham surfaces \cite{zbMATH03305711} 
\begin{equation}
\label{equation:BPsurface}
\mathcal{S}_{\alpha_{i}}
:=
\{y^{\alpha_{1}}_{1}+y^{\alpha_{2}}_{2}+y^{\alpha_{3}}_{3}=0\},
\end{equation}
for pairwise coprime integers $\alpha_{i}\geq 2$, 
we form the smooth affine modification 
\begin{equation}
\label{equation:4dimensional}
\mathcal{Y}(n,\alpha_{i},p)
:=
\{x^{n}z=
y^{\alpha_{1}}_{1}+y^{\alpha_{2}}_{2}+y^{\alpha_{3}}_{3}+xp(x,y_{1},y_{2},y_{3})\},
\end{equation} 
where $n\geq 2$, and $p(x,y_{1},y_{2},y_{3})\in F[x,y_{1},y_{2},y_{3}]$ satisfies $p(0,0,0,0)\in F^{\ast}$.
The $4$-fold \eqref{equation:4dimensional} is not an affine space since it has a nontrivial Makar-Limanov invariant (see e.g., \cite[Proposition 11.1]{MR2327241}).
In a special case we show \eqref{equation:4dimensional} is "stably $\AA^{1}$-contractible" in the following sense:
\begin{theorem}
\label{thm:4dimensional-A1Cont}
If $\alpha_{3}=m\alpha_{1}\alpha_{2}+1$, $m>0$, then the smooth affine $4$-fold $\mathcal{Y}(n,\alpha_{i},p)$
in \eqref{equation:4dimensional} becomes $\AA^{1}$-contractible after a finite suspension with the projective line $\PP^{1}$ pointed at infinity:
\[
\mathcal{Y}(n,\alpha_{i},p)\wedge \PP^{1}\wedge \cdots \wedge \PP^{1}\sim_{\AA^{1}}\ast.
\]
\end{theorem}

\subsection*{Relation to other works}
Asok-Doran produced families of non-isomorphic smooth quasi-affine but not affine $n$-dimensional varieties for $n\geq 4$ that are $\AA^{1}$-contractible \cite{MR2335246}.
Their main technique,
tracing back to classical contractibility results for PL and smooth manifolds, 
is to form quotients of affine spaces by free actions of unipotent groups.
Inspired by Asok's lectures \cite{AsokOttawa} and connections to the Zariski cancellation problem, 
\cite{MRDF} and \cite{MR3549169} established $\AA^{1}$-contractibility results for smooth affine Koras-Russell $3$-folds of the first and second kind. 
In this paper we open a systematic study of such examples by utilizing the geometric technique of affine modifications to craft accessible and widely applicable proofs of $\AA^{1}$-contractibility. 
This represents a concise version of the theory that strengthens the results in \cite{MRDF} and \cite{MR3549169}, presents new insights, and opens new lines of inquiry.

\subsection*{Outline}
In \aref{section:background} we review background on affine algebraic geometry and $\AA^{1}$-homotopy theory.
We review and give examples of affine modifications. 
Related to $\AA^{1}$-contractibility we study $\AA^{1}$-chain connectedness  
--- an algebro-geometric analogue of path connectedness introduced by Asok-Morel \cite{MR2803793} ---
which loosely speaking means that any two points can be connected by the images of a chain of maps from the affine line.
There are many elementary open problems related to these notions;
we record a few sample questions which we believe warrant further investigation.

\aref{section:hcgc} improves on the stable $\AA^{1}$-contractibility results in \cite{MR3549169} by producing several new families of examples.
This is achieved by explicit calculations of Bloch's higher Chow groups \cite{zbMATH03983347} or equivalently of motivic cohomology in the sense of 
Suslin-Voevodsky \cite{zbMATH01445142}, \cite{zbMATH01724017}.
Using the techniques of stable motivic homotopy theory allow us to conclude $\AA^{1}$-contractibility after a finite number of suspensions with $\PP^{1}$.  
For an elucidation of these techniques and further references we refer to \cite{MR3549169}.

Our main results are shown in \aref{section:unstable}.
The proof of \aref{thm:KRBundle-Stable-A1Cont} proceeds by induction on the number of singular fibers of \eqref{equation:KRfiberbundleprojection}.  
For the induction step we invoke the geometric construction of an affine modification along a divisor in the sense of Kaliman-Za{\u\i}denberg \cite{zbMATH01286199},
a notion tracing back to \cite{MR0206035} and \cite{MR0008468}.
Moreover,
by homotopy purity for closed embeddings \cite[Theorem 3.2.23]{MR1813224} and excision over open neighborhoods of the origin, 
we reduce the proof to the special case of deformed Koras-Russell $3$-folds of the first kind \eqref{equation:firstkindpolynomial}.
As alluded to above an explicit calculation shows $\mathcal{X}(n,\alpha_{i},p)\to\Spec(F)$ induces an isomorphism on higher Chow groups;
this generalizes \cite[Proposition 3.3]{MR3549169}. 
Combined with the results in \cite[\S4]{MR3549169} we conclude there is an $\AA^{1}$-weak equivalence:
\begin{equation}
\label{equation:stableA1KRbundles}
\mathcal{X}(n,\alpha_{i},p)\wedge \PP^{1}\wedge \cdots \wedge \PP^{1}\sim_{\AA^{1}}\ast.
\end{equation}

To understand the geometry of the $3$-fold $\mathcal{X}(n,\alpha_i,p)$  we study its $\GG_a$-action corresponding to the locally nilpotent derivation 
$$
\partial
=
x^{n}\frac{\partial}{\partial y}+(\alpha_{1}y^{\alpha_{1}-1}+x\frac{\partial}{\partial y}p(x,y,t))\frac{\partial}{\partial z}
$$ 
on the coordinate ring of $\mathcal{X}(n,\alpha_i,p)$, 
with fixed point locus the affine line $\{x=y=t=0\}\cong\AA^{1}_z$. 
The geometric quotient $\mathcal{X}(n,\alpha_i,p)\to \mathcal{X}(n,\alpha_i,p)/\GG_a$ yields an 
$\AA^{1}$-bundle $\mathcal{X}(n,\alpha_i,p)\setminus \AA^{1}_z\to\mathfrak{S}(\alpha_1,\alpha_2)$ in the category of algebraic spaces \cite{zbMATH03350017}.
We note there exists a factorization
\begin{equation}
\label{equation:asfactorization}
\xymatrix{
\mathcal{X}(n,\alpha_{i},p)\setminus \AA^{1}_{z}  \ar[dr] \ar[rr]^-{\pi_{\vert}} & & \AA^{2}_{x,t}\setminus\{(0,0)\} \\
& \mathfrak{S}(\alpha_{1},\alpha_{2}), \ar[ur] &
}
\end{equation}
where $\pi_{\vert}$ is the restriction of $\pi:=\pr_{x,t}:\mathcal{X}(n,\alpha_i,p)\to \AA^{2}_{x,t}$ to $\mathcal{X}(n,\alpha_i,p)\setminus \AA^{1}_z$.
To construct \eqref{equation:asfactorization} we form a cyclic Galois cover of $\AA^{2}_{x,t}$ of order $\alpha_{1}$ and consequently of $\mathcal{X}(n,\alpha_{i},p)$ by pullback via $\pi$.
The maps arise as geometric quotients for $\mu_{\alpha_{1}}$-equivariant maps by gluing copies of $\AA^{2}\setminus\{(0,0)\}$ by the identify away from a family of cuspidal curves.
Here $\mathcal{X}(n,\alpha_{i},p)\setminus \AA^{1}_{z}\to\mathfrak{S}(\alpha_{1},\alpha_{2})$ is an \'etale locally trivial $\AA^{1}$-bundle.
We have $\mathfrak{S}(\alpha_{1},\alpha_{2})\cong\mathfrak{S}(\alpha_{1},1)$, 
and both of the projection maps for the smooth quasi-affine $4$-fold 
\begin{equation}
\label{equation:asfiberproduct}
(\mathcal{X}(n,\alpha_{i},p)\setminus \AA^{1}_{z})\times_{\mathfrak{S}(\alpha_{1},\alpha_{2})} (\mathcal{X}(n,\alpha_{1},1,p)\setminus \AA^{1}_{z})
\end{equation}
are Zariski  locally trivial $\AA^{1}$-bundles, 
and hence $\AA^{1}$-weak equivalences.
The fiber product in \eqref{equation:asfiberproduct} is formed in algebraic spaces over the punctured affine plane $\AA^{2}_{x,t}\setminus\{(0,0)\}$.
Furthermore, 
the projection $\pr_{x}:\mathcal{X}(n,\alpha_{1},1,p)\to\AA^{1}_{x}$ is a trivial $\AA^{2}$-bundle.
Hence $\mathcal{X}(n,\alpha_{1},1,p)\cong\AA^{3}_{x,y,z}$ and we obtain: 
\[
\mathcal{X}(n,\alpha_{i},p)\setminus \AA^{1}_{z}
\sim_{\AA^{1}}
\mathcal{X}(n,\alpha_{1},1,p)\setminus \AA^{1}_{z}
\sim_{\AA^{1}}
\AA^{2}_{x,t}\setminus\{(0,0)\}.
\]
Combined with an explicit calculation of cohomology with Milnor-Witt $K$-theory sheaves using \eqref{equation:stableA1KRbundles} we conclude $\pi_{\vert}$ is an $\AA^{1}$-weak equivalence.

The projection $\pi$ gives rise to a commutative diagram of cofiber sequences:
\begin{equation}
\label{equation:excisiondiagram}
\xymatrix{
\mathcal{X}(n,\alpha_{i},p)\setminus \AA^{1}_{z}  \ar[r]\ar[d]_{\pi_{\vert}} & 
\mathcal{X}(n,\alpha_{i},p) \ar[r]\ar[d]_{\pi} &
\mathcal{X}(n,\alpha_{i},p)/(\mathcal{X}(n,\alpha_{i},p)\setminus \AA^{1}_{z})\sim_{\AA^{1}} (\AA^{1}_{z})_{+}\wedge(\PP^{1})^{\wedge2}  \ar[d]_{\sim_{\AA^{1}}} \\ 
\AA^{2}_{x,t}\setminus\{(0,0)\}  \ar[r] & \AA^{2}_{x,t} \ar[r] & \AA^{2}_{x,t}/\AA^{2}_{x,t}\setminus\{(0,0)\}\sim_{\AA^{1}}(\PP^{1})^{\wedge2}.
}
\end{equation}
The rightmost vertical map in \eqref{equation:excisiondiagram} is an $\AA^{1}$-weak equivalence according to homotopy purity for closed embeddings \cite[Theorem 3.2.23]{MR1813224}.
By continuing \eqref{equation:excisiondiagram} into a long cofiber sequence we find that the simplicial suspension of $\mathcal{X}(n,\alpha_{i},p)$ is $\AA^{1}$-contractible:
\begin{equation}
\label{equation:simplicialcontractible}
\Sigma_{s}\mathcal{X}(n,\alpha_{i},p)
:=
\mathcal{X}(n,\alpha_{i},p)\wedge S^{1}\sim_{\AA^{1}}\ast.
\end{equation}
This finishes the proof of \aref{thm:KRBundle-Stable-A1Cont} since $\PP^{1}\sim_{\AA^{1}}\GG_{m}\wedge S^{1}$ by the standard open covering of the projective line.
\medskip

In conclusion we note a family of $\mathbb{A}^{1}$-contractible $3$-folds with an $\mathbb{A}^{1}$-contractible total space.
Let us fix $n\geq4$, 
coprime integers $\alpha_{1}$, $\alpha_{2}\geq2$, 
and set $\mathcal{Y}:=\mathrm{Spec}(F[a_{2},\ldots,a_{n-1}])\simeq\mathbb{A}_{F}^{n-2}$.
With these conventions we define the subscheme $\mathfrak{X}\subset \mathcal{Y}\times\mathbb{A}_{F}^{4}$ by the equation: 
\begin{equation}
\label{equation:smoothfamily}
\mathfrak{X}
:=
\{x^{n}z
=
y^{\alpha_{1}}+t^{\alpha_{2}}+x+x^{2}+\sum_{i=2}^{n-1}a_{i}x^{i+1}\}.
\end{equation}
Here $\pr_{\mathcal{Y}}:\mathfrak{X}\to \mathcal{Y}$ is a smooth family whose fibers are $\mathbb{A}^{1}$-contractible and non-isomorphic to $\AA^{3}$ over 
the corresponding residue fields \cite{MRDF}. 
Furthermore, 
the fibers of $\pr_{\mathcal{Y}}$ over the $F$-rational points of $\mathcal{Y}$ are pairwise non-isomorphic $F$-schemes which are all $\mathbb{A}^{1}$-stably isomorphic. 
We claim the total space $\mathfrak{X}$ of this family is $\mathbb{A}^{1}$-contractible. 
Indeed, 
rewriting the defining equation \eqref{equation:smoothfamily} for $\mathfrak{X}$ on the form 
\[
\{x^{3}(-a_{2}-\sum_{i=3}^{n-1}a_{i}x^{i-2}+x^{n-3}z)=y^{\alpha_{1}}+t^{\alpha_{2}}+x(1+x)\},
\]
and noting that $Z=-a_{2}-\sum_{i=3}^{n-1}a_{i}x^{i-2}+x^{n-3}z$ is a variable of $F[x][a_{2},a_{3},\ldots,a_{n-1},z]$ in the sense that $F[x][a_{2},a_{3},\ldots,a_{n-1},z]=F[x][Z,a_{3},\ldots,a_{n-1},z]$,
we see that $\mathfrak{X}$ is isomorphic to the product of $\mathbb{A}_{F}^{n-2}=\mathrm{Spec}(F[a_{3},\ldots,a_{n-1},z])$ with the deformed Koras-Russell threefold: 
\[
\mathcal{X}(3,\alpha_{1},\alpha_{2},(1+x))
:=
\{x^{3}Z=y^{\alpha_{1}}+t^{\alpha_{2}}+x(1+x)\subset\mathrm{Spec}\left(F[x,Z,y,t]\right)\}.
\]
Since $\mathcal{X}(3,\alpha_{1},\alpha_{2},(1+x))$ is $\mathbb{A}^{1}$-contractible \cite{MRDF},
so is $\mathfrak{X}$ by homotopy invariance.
\medskip

Throughout the paper we work over an algebraically closed field $F$ of characteristic zero.  
We write $\Sch_{F}$ and $\Sm_{F}$ for the categories of separated and separated smooth schemes of finite type over $F$.

\section{Background in affine algebraic geometry and $\AA^{1}$-homotopy theory}
\label{section:background}

\subsection{Affine modifications}
\label{subsection:affinemodifications}
Our main technique for generating examples of $\AA^{1}$-contractible smooth affine schemes involves affine modifications.
This construction is ubiquitous, 
e.g., 
every affine birational map between integral affine schemes of finite type arises from an affine modification see \cite[Theorem 1.1]{zbMATH01286199}. 
An affine modification is a birational construction which also makes sense for non-affine schemes \cite{Duboulozaffinemodification}:

\begin{defn}
\label{defn:affinemodification}
Suppose $(\mathcal{Z}\subset \mathcal{D}\subset \mathcal{X})$ be a triple consisting of an effective Cartier divisor $\mathcal{D}$ on an integral scheme $\mathcal{X}$ with sheaf 
of rational functions $\mathcal{K}_{\mathcal{X}}$ and a closed subscheme $\mathcal{Z}$ with ideal sheaf  $\II_{\mathcal{Z}}\subset\mathscr{O}_{\mathcal{X}}(-\mathcal{D})$. 
Then the affine modification of $\mathcal{X}$ along $\mathcal{D}$ with center in $\mathcal{Z}$ is the affine $\mathcal{X}$-scheme 
$$
\pi_{\mathcal{D},\mathcal{Z}}:\widetilde{\mathcal{X}}(\mathcal{D},\mathcal{Z})=\Spec(\mathcal{O}_{\mathcal{X}}[\II_{\mathcal{Z}}/\mathcal{D}])\to \mathcal{X},
$$
where $\mathcal{O}_{\mathcal{X}}[\II_{\mathcal{Z}}/\mathcal{D}]$ denotes the quotient of the Rees algebra 
$$
\mathcal{R}(\II_{\mathcal{Z}}\otimes\mathcal{O}_{\mathcal{X}}(\mathcal{D}))
=
\bigoplus_{n\geq0}(\II_{\mathcal{Z}}\otimes\mathcal{O}_{\mathcal{X}}(\mathcal{D}))^{n}t^{n}\subset\mathcal{K}_{\mathcal{X}}[t]
$$ 
of the fractional ideal $\II_{\mathcal{Z}}\otimes\mathcal{O}_{\mathcal{X}}(\mathcal{D})\subset\mathcal{K}_{\mathcal{X}}$ by the ideal generated by $1-t$.
\end{defn}

Via the canonical open immersion 
\[
j:
\tilde{\mathcal{X}}(\mathcal{D},\mathcal{Z})
\hookrightarrow
\mathrm{Proj}_{\mathcal{X}}\left(\mathcal{R}(\II_{\mathcal{Z}}\otimes\mathcal{O}_{\mathcal{X}}(\mathcal{D}))\right)
\simeq
\mathrm{Proj}_{\mathcal{X}}\left(\mathcal{R}(\II_{\mathcal{Z}}\right)),
\]
the affine modification $\widetilde{\mathcal{X}}(\mathcal{D},\mathcal{Z})$ coincides with the complement in the blow-up 
\[
\sigma_{\mathcal{Z}}:\Bl_{\mathcal{Z}}\mathcal{X}
=
\mathrm{Proj}_{\mathcal{X}}\left(\mathcal{R}(\II_{\mathcal{Z}})\right)\to \mathcal{X}
\] 
of $\mathcal{X}$ with center $\mathcal{Z}$ of the proper transform
$$
\mathcal{D}^{pr}:=
\mathrm{Proj}_{\mathcal{X}}\left(\mathcal{R}(\II_{\mathcal{Z}}\otimes\mathcal{O}_{\mathcal{X}}(\mathcal{D}))/t\right)\subset\Bl_{\II_{\mathcal{Z}}}\mathcal{X}
$$ 
of the divisor $\mathcal{D}$. 
The proper transform $\mathcal{D}^{pr}$ of $\mathcal{D}$ is in general different from its strict transform $\mathcal{D}'$, 
i.e., 
the closure of $\sigma^{-1}_{\mathcal{Z}}(\mathcal{D}\setminus \mathcal{Z}))$ in $\Bl_{\mathcal{Z}}\mathcal{X}$.
However, 
these two schemes coincide when $\mathcal{Z}$ is a local complete intersection in $\mathcal{D}$ \cite[Proposition 1.10]{Duboulozaffinemodification}.

Letting $\mathcal{E}:=\sigma_{\mathcal{Z}}^{-1}(\mathcal{Z})$ be the exceptional divisor of the blow-up $\sigma_{\mathcal{Z}}$, 
we note that $\sigma_{\mathcal{Z}}$ restricts to an isomorphism $\widetilde{\mathcal{X}}(\mathcal{D},\mathcal{Z})\setminus \mathcal{E} \overset{\cong}{\to}  \mathcal{X}\setminus \mathcal{D}$. 
We refer to 
\[
\mathcal{E}_{\mathcal{Z}/\mathcal{D}}:=\widetilde{\mathcal{X}}(\mathcal{D},\mathcal{Z})\cap \mathcal{E}=\mathcal{E}\setminus \mathcal{D}^{pr}
\]
as the exceptional fiber of the affine modification $\pi_{\mathcal{D},\mathcal{Z}}:\widetilde{\mathcal{X}}(\mathcal{D},\mathcal{Z})\to \mathcal{X}$.

For an integral affine scheme $\mathcal{X}=\Spec(A)$ and $f\neq 0$ an element in an ideal $\mathcal{I}\subset A$, 
the affine modification of $\mathcal{X}$ along the principal divisor $\ddiv(f)$ with center $V(\mathcal{I})$ is isomorphic to the spectrum of the sub-$A$-algebra 
\[
A[\mathcal{I}/f]
\cong
A[t\mathcal{I}]/(1-tf)
\cong
\{a/f^{k}\,\vert\, a\in \mathcal{I}^{k}, k\geq 0\}
\subset
A_{f}.
\] 
as defined in \cite{zbMATH01286199}.
We denote the exceptional divisor of an affine modification of an integral affine scheme by $\mathcal{E}_{\mathcal{I}/f}$.

\begin{example} (See \cite{MR0206035}.)
Suppose $f,a_{1},\dots,a_{n}$ is a regular sequence generating the ideal of the closed subscheme $\mathcal{Z}\subset\AA^{m}=\Spec(F[x_{1},\dots,x_{m}])$.
Then $\widetilde{\AA^{n}}(\ddiv(f),\mathcal{Z})$ is isomorphic to the subscheme of $\AA^{m+n}$ defined by the equations:
\[
\{a_{i}(x_{1},...,x_{m})-t_if(x_{1},...,x_{m})=0\}_{1\leq i\leq n}
\subset 
\AA^{m+n}
=
\Spec(F[t_1,\dots,t_n,x_{1},\dots,x_{m}]).
\]
\end{example}
\begin{example}
The following smooth schemes from \aref{sec:intro} are examples of affine modifications.
Here we identify the coordinate ring of $\mathcal{X}(n,\alpha_{i},p)$ with the subalgebra $F[x,y,t,x^{-n}(y^{\alpha_{1}}+t^{\alpha_{2}}+xp(x,y,t))]$ of $F[x^{\pm 1},y,t]$, 
i.e., 
$\mathcal{X}(n,\alpha_{i},p)$ is the affine modification $\pr_{x,y,t}: \mathcal{X}(n,\alpha_{i},p)\to\AA^{3}_{x,y,t}$ of $\Spec(F[x,y,t])$ along the principal divisor $\ddiv(x^{n})$ 
with center $\mathcal{Z}$ defined by the ideal $(x^{n},y^{\alpha_{1}}+t^{\alpha_{2}}+xp(x,y,t))$.
That is, 
$\mathcal{X}(n,\alpha_{i},p)$ is isomorphic to the complement of the proper transform of the divisor $n\{x=0\}$ in the blowup of $\AA^{3}$ with center $\mathcal{Z}$.
Similar reasoning applies to the other examples in \aref{table:1}.
\vspace{-0.1in}
\begin{table}[h!]
\begin{center}
\begin{tabular}{ | l | | l | l | l | }
\hline
$\widetilde{\mathcal{X}}(\mathcal{D},\mathcal{Z})$  &  $\mathcal{X}$ & divisor $\mathcal{D}$ & ideal defining $\mathcal{Z}$ \\ \hline \hline
\eqref{equation:firstkindpolynomial}  $\mathcal{X}(n,\alpha_{i},p)$ & $\AA^{3}$ & $n\{x=0\}$ & $(x^{n},y^{\alpha_{1}}+t^{\alpha_{2}}+xp(x,y,t))$ \\ \hline
\eqref{equation:mdegeneratefibers} $\mathcal{X}_{m}(n_{i},\alpha_{j},\beta_{j},a)$  & $\AA^{3}$ & $\sum_{i}n_{i}\ddiv(l_{i}(x))$
                                                          & $(\prod_{i} l_{i}(x)^{n_i},\sum_i((\prod_{j\neq i} l_{j}(x))(y^{\alpha_i}+t^{\beta_i}))+\prod_{i} l_{i}(x))$\\ \hline
\eqref{equation:KRfiberbundleprojection}  $\mathcal{X}(s,R)$      & $\AA^{3}$ & $\ddiv(s(x))$ & $(s(x), R(x,y,t)$)\\ \hline
\eqref{equation:iteratedfirstkindpolynomial}  $\mathcal{Y}(m,n_{i},\alpha_{i},p)$   & \eqref{equation:firstkindpolynomial} & $n_{2}\{x=0\}$ & $(x^{n_{2}},y-z^{m})$\\ \hline
\eqref{equation:4dimensional} $\mathcal{Y}(n,\alpha_{i},p)$ & $\AA^{4}$  & $n\{x=0\}$ & $(x^{n},y_{1}^{\alpha_{1}}+y_{2}^{\alpha_{2}}+y_{3}^{\alpha_{3}}+xp(x,y_{1},y_{2},y_{3}))$\\
\hline
\end{tabular}
\vspace{0.1in}
\caption{Affine modifications}
\label{table:1}
\end{center}
\end{table}
\end{example}
\vspace{-0.3in}

Given a closed embedding $\mathcal{Z}\hookrightarrow\mathcal{X}$ between smooth schemes,
the (affine) deformation space $\mathrm{D}(\mathcal{X},\mathcal{Z})$ of $\mathcal{Z}$ in $\mathcal{X}$ \cite[Chapter 5]{zbMATH01027930}, \cite{zbMATH01016564} is the complement of the 
proper transform of $\mathcal{X}\times\left\{ 0\right\}$ in the blow-up $\mathrm{Bl}_{\mathcal{Z}\times\left\{ 0\right\} }(\mathcal{X}\times\mathbb{A}^{1})\to \mathcal{X}\times\mathbb{A}^{1}$ of 
$\mathcal{X}\times\mathbb{A}^{1}$ along $\mathcal{Z}\times\left\{ 0\right\} $. 
In other words, 
$\mathrm{D}(\mathcal{X},\mathcal{Z})$ is the affine modification of $\mathcal{X}\times\mathbb{A}^{1}$ along the divisor $\mathcal{D}=X\times\left\{ 0\right\} $ with center in 
$\mathcal{Z}=Z\times\left\{ 0\right\} $. 
When $X$ is affine, 
this construction is called the hyperbolic modification of $\mathcal{X}$ with center $\mathcal{Z}$ (see \cite{MR1734345}).

\subsection{Exotic affine spaces}
Recall that a smooth affine $d$-dimensional $\CC$-variety $\mathcal{X}$ is called exotic if $\mathcal{X}$ is non-isomorphic to $\AA^{d}_{\CC}$ and its underlying smooth manifold is 
diffeomorphic to $\RR^{2d}$ \cite[Definition 3.1]{MR1734345}.
In dimension $d\ge 3$, 
the $h$-cobordism and Lefschetz hyperplane section theorems imply that smooth affine contractible $\CC$-varieties are diffeomorphic to $\RR^{2d}$ by the Dimca-Ramanujam theorem 
\cite[\S3]{MR1734345}.
The following beautiful result due to Kaliman \cite{MR1895930} is a useful test for detecting exotic $3$-folds.
It was extended to all fields of characteristic zero by Daigle-Kaliman \cite[Theorem 4.2]{MR2567148} using a Lefschetz principle.

\begin{theorem}
\label{theorem:Kalimancriterion}
If the general fibers of a regular function $\phi\colon\AA^{3}\to\AA^{1}$ are isomorphic to the affine plane $\AA^{2}$,
then all the closed fibers of $\phi$ are isomorphic to $\AA^{2}$.
\end{theorem}

\begin{example}
\aref{theorem:Kalimancriterion} does not apply to the $4$-fold $\mathcal{Y}(n,\alpha_i,p)$ defined in \eqref{equation:4dimensional} (see \cite{MR1817379}).
\end{example}

Kaliman-Za{\u\i}denberg noted the following version of Miyanishi's characterization of $\AA^{3}_{\CC}$ \cite[p.1650]{MR1817379}.
\begin{theorem}
\label{theorem:miyanishi's characterization}
A smooth contractible affine complex $3$-fold $\mathcal{X}$ is isomorphic to $\AA^{3}_{\CC}$ if and only if there exists a regular function $\phi:\mathcal{X}\to\AA^{1}_{\CC}$ 
such that every closed fiber has at most isolated singularities and the general fibers are isomorphic to $\AA^{2}_{\CC}$.
\end{theorem}

\subsection{Topological contractibility of affine modifications}
In \cite[Proposition 3.1, Theorem 3.1]{zbMATH01286199} Kaliman-Zaidenberg give useful criteria for when an affine modification induces isomorphisms on fundamantal groups and singular homology.
When combined with the classical Hurewicz and Whitehead theorems (see e.g., \cite[Theorems 4.5, 4.23]{zbMATH02103273}), 
one obtains the following version of \cite[Corollary 3.1, Remark 3.1]{zbMATH01286199}.
\begin{theorem}
\label{theorem:Kaliman-Zaidenbergcriterion}
Suppose $(\mathcal{Z}=V(\mathcal{I})\subset \mathcal{D}_{f}=\ddiv(f)\subset \mathcal{X}=\Spec(A)$ is a triple in $\Sch_{\CC}$, 
where $\mathcal{X}$ is an affine variety and $\mathcal{I}$ is generated by a regular sequence $f,a_{1},\dots,a_{n}$. Let $\pi:\widetilde{\mathcal{X}}(\mathcal{D}_{f},\mathcal{Z})\to \mathcal{X}$ be the affine modification of $\mathcal{X}$ along $\mathcal{D}$ with center $\mathcal{Z}$ and suppose that:
\begin{itemize}
\item[(1)] the supports of $\mathcal{D}_{f}$ and the exceptional divisor $\mathcal{E}_{\mathcal{I}/f}$ are irreducible,
\item[(2)] $\pi^{\ast}(\supp \mathcal{D}_{f})=\mathcal{E}_{\mathcal{I}/f}$,
\item[(3)] $\supp \mathcal{D}_{f}$  and $\mathcal{E}_{\mathcal{I}/f}$ are topological manifolds,
\item[(4)] $\mathcal{Z}\hookrightarrow \mathcal{D}_{f}$ is a homotopy equivalence.
\end{itemize}
Then $\widetilde{\mathcal{X}}(\mathcal{D}_{f},\mathcal{Z})$ is contractible if and only if $\mathcal{X}$ is contractible.
\end{theorem}
\begin{remark}
\label{remark:topologicalirreduciblecomponents}
Assumption (2) implies $\sigma_{\mathcal{I}}(\mathcal{E}_{\mathcal{I}/f})$ intersects the smooth part of $\mathcal{D}_{f}$ nontrivially. 
\aref{theorem:Kaliman-Zaidenbergcriterion} holds more generally for finite decompositions into irreducible components $\mathcal{D}_{f}=\sum_{i=1}^{m} \mathcal{D}_{i}$ 
and $\mathcal{E}_{\mathcal{I}/f}=\sum_{i=1}^{m} \mathcal{E}_{i}$ such that (2)-(4) hold componentwise.
\end{remark}

\begin{example}
The hypersurface $\mathcal{S}:=\{x^{2}z=y^{2}+x\}\subset\AA^{3}$ is the affine modification of $\AA^{2}$ along the principal divisor $\mathcal{D}_f$ given by $f=x^2$ with center defined by $(x^{2},y^{2}+x)$. 
It is not topologically contractible since the first homology group $H_{1}(\mathcal{S})\cong\ZZ/2\ZZ$. 
This follows from the Thom isomorphism and the commutative diagram of exact sequences: 
\begin{equation}
\label{equation:chowgroupaffinemodification2}
\xymatrix{
H_2(\mathcal{S}) \ar[r]\ar[d] & 
H_0(\mathcal{E}_{\mathcal{I}/f})=\ZZ \ar[r]\ar[d]_{\cdot 2} &
H_1(\mathcal{S}\setminus \mathcal{E}_{\mathcal{I}/f})  \ar[r]\ar[d]_{\cong}&  
H_1(\mathcal{S})  \ar[r]\ar[d] & 
0\\ 
H_2(\AA^{2})=0  \ar[r] & 
H_0(\mathcal{D}_f)=\ZZ \ar[r]^{\cong} &
H_1(\AA^{2}\setminus \mathcal{D}_f)  \ar[r]&  
H_1(\AA^{2})=0  \ar[r] & 
0. \\ 
}
\end{equation}
This example does not contradict \aref{theorem:Kaliman-Zaidenbergcriterion} because $\pi^{*}(\supp \mathcal{D}_f)=2\mathcal{E}_{\mathcal{I}/f}\neq \mathcal{E}_{\mathcal{I}/f}$.
\end{example}

\begin{example}
\aref{theorem:Kalimancriterion} and \aref{theorem:Kaliman-Zaidenbergcriterion} show the following examples are exotic $3$-folds.
\begin{itemize}
\item 
Every deformed Koras-Russell $3$-fold $\mathcal{X}(n,\alpha_{i},p)$ given by  \eqref{equation:firstkindpolynomial} is topologically contractible. 
The projection $\pr_{x}\colon \mathcal{X}(n,\alpha_{i},p)\to \AA^{1}_{\CC}$ \eqref{equation:firstkindpolynomial} has zero fiber 
$\pr_{x}^{-1}(0)=\Gamma_{\alpha_{1},\alpha_{2}}\times \AA^{1}_{z}\not\iso\AA^{2}_{\CC}$ \eqref{equation:firstkindpolynomialinverse} with non-isolated singularities.
Hence $\mathcal{X}(n,\alpha_{i},p)$ is not isomorphic to $\AA^{3}$.
\item 
By \aref{remark:topologicalirreduciblecomponents}, 
$\mathcal{X}_{m}(n_{i},\alpha_{j},\beta_{j},a)$ defined in \eqref{equation:mdegeneratefibers} is topologically contractible. 
The projection $\pr_{x}: \mathcal{X}_{m}(n_{i},\alpha_{j},\beta_{j},a)\to \AA^{1}_{x}$ has general fiber $\AA^{2}$ while $\pr_{x}^{-1}(x_{j})\cong \AA^{1}\times \Gamma_{\alpha_{j},\beta_{j}}$.
\item 
As noted in \aref{remark:KRaffinespace} a Koras-Russell fiber bundle is isomorphic to $\AA^{3}$ if and only if for every zero of $s(x)$ the curve $\{a(y,t)=0\}$ is isomorphic to $\AA^{1}$.
\item 
$\mathcal{Y}(m,n_{i},\alpha_{i},p)$ defined in \eqref{equation:iteratedfirstkindpolynomial} is topologically contractible because it is an affine modification of some topologically contractible deformed 
Koras-Russell $3$-fold. 
Here $\pr_{x}:\mathcal{Y}(m,n_{i},\alpha_{i},p)\to \AA^{1}_{x}$ has general fiber $\AA^{2}$ while $\pr_{x}^{-1}(0)\cong \AA^{1}\times \Gamma_{m\alpha_{1},\alpha_{2}}$.
\end{itemize}
We note every vector bundle on these exotic $3$-folds are trivial by a result of Murthy \cite[Theorem 3.6]{zbMATH01983932}. 
\end{example}

\subsection{Locally nilpotent derivations}
The Makar-Limanov invariant of an affine algebraic variety $\mathcal{X}$ is the subring $\ML(\mathcal{X})$ of $\Gamma(\mathcal{X},\mathcal{O}_{\mathcal{X}})$ comprised of regular functions invariant 
under any $\GG_{a}$-action \cite{zbMATH00997597}. 
For a commutative $F$-domain $A$, 
let $\operatorname{LND}(A)$ denote the locally nilpotent derivations $\partial:A\to A$.
With this notation we have: 
\[
\ML(\mathcal{X})=\bigcap_{\partial \in \operatorname{LND}(\Gamma(\mathcal{X},\mathcal{O}_{\mathcal{X}}))}\ker(\partial).
\]
Moreover, 
the Derksen invariant $\operatorname{D}(\mathcal{X})$ of $\mathcal{X}$ is the subalgebra of $\Gamma(\mathcal{X},\mathcal{O}_{\mathcal{X}})$ generated by the kernels of all nonzero 
locally nilpotent derivations $\partial\in \operatorname{LND}(\Gamma(\mathcal{X},\mathcal{O}_{\mathcal{X}}))$ (see \cite{MR2259515} for a survey of these invariants).
\begin{example}
\label{example:MLandDerksenInvariant}
\begin{itemize}
\item
$\ML(\AA^{n})=F$ for $n\ge 1$ and $\operatorname{D}(\AA^{n})=F[x_{1},\dots,x_n]$ for $n\ge 2$ while $\operatorname{D}(\AA^{1})=0$.
\item
The same arguments as in \cite[Theorem 9.1]{MR2327241} for the Russell cubic, 
i.e., 
the hypersurface 
\begin{equation}
\label{equation:russell3fold}
\mathcal{R}
:=
\{x^{2}z=y^{2}+t^{3}+x\}\subset\AA^{4},
\end{equation}
show more generally that $\ML(\mathcal{X}(s,R))=F[x]$ and $\operatorname{D}(\mathcal{X}(s,R))=F[x,z,t]$ for the Koras-Russell fiber bundles introduced in \aref{def:KRBundle}.
\end{itemize}
\end{example}

\begin{lemma}
\label{lemma:mnonisomorphic}
For $m\neq m'\in\ZZ$ the smooth affine $3$-folds $\mathcal{X}_{m}$ and $\mathcal{X}_{m'}$ with degenerate fibers \eqref{equation:mdegeneratefibers} are non-isomorphic.
\end{lemma}
\begin{proof}
If there exists an isomorphism $\phi:\mathcal{X}_{m}\to \mathcal{X}_{m'}$ we obtain a commutative diagram:
\begin{equation}
\label{equation:MLfactorization}
\xymatrix{
\mathcal{X}_{m}  \ar[d]_-{} \ar[r]^-{\phi} & \mathcal{X}_{m'} \ar[d]^-{} \\
\Spec(\ML(\mathcal{X}_{m})) \ar[r]^-{\phi_{\vert}} & \Spec(\ML(\mathcal{X}_{m'})).
}
\end{equation}
Here the vertical maps, 
induced by $\ML(\mathcal{X}_{m})\subset \mathcal{O}_{\mathcal{X}_{m}}$ and $\ML(\mathcal{X}_{m'})\subset \mathcal{O}_{\mathcal{X}_{m'}}$ can be identified with the respective 
projection maps to $\AA^{1}_{x}$ because $\ML(\mathcal{X}_{m})=\ML(\mathcal{X}_{m'})=F[x]$ according to \aref{example:MLandDerksenInvariant}.
Since $m\neq m'$ we conclude by counting the number of closed fibers non-isomorphic to $\AA^{2}$.
\end{proof}

\subsection{$\AA^{1}$-homotopy theory}
Following Morel-Voevodsky \cite{MR1813224} (see e.g., \cite{zbMATH05116082} for an introduction) we view schemes over $F$ as analogous to topological spaces with $\AA^{1}$ 
playing the role of the unit interval.
Recall that $\mathcal{X}$ is $\AA^{1}$-contractible if the canonical map $\mathcal{X}\to\Spec(F)$ is an $\AA^{1}$-weak equivalence in the sense of \cite{MR1813224}.
Affine spaces are readily $\AA^{1}$-contractible.
We will make repeatedly use of the following example.

\begin{example}
The cuspidal curve $\Gamma_{\alpha_{1},\alpha_{2}}$ is $\AA^{1}$-contractible.
Since $\mathcal{O}_{\Gamma_{\alpha_{1},\alpha_{2}}}$ is isomorphic to $F[w^{\alpha_{1}},w^{\alpha_{2}}]$, 
the normalization of the cuspidal curve $\AA^{1}_{w}\to\Gamma_{\alpha_{1},\alpha_{2}}$ is given by $w\mapsto (w^{\alpha_{2}},w^{\alpha_{1}})$.
This map is an $\AA^{1}$-weak equivalence, 
even an isomorphism of representable presheaves on $\Sm_{F}$ (see \cite[Example 2.1]{MR2335246}).
\end{example}

The following result elucidates the role of affine modifications in $\AA^{1}$-homotopy theory.

\begin{theorem}
\label{theorem:A1affinemodifications}
Suppose $(\mathcal{Z} \subset \mathcal{D} \subset \mathcal{X})$ is a triple in $\Sm_{F}$ as in \aref{defn:affinemodification} with affine modification 
$\pi:\widetilde{\mathcal{X}}(\mathcal{D},\mathcal{Z})\to \mathcal{X}$ such that the following holds:
\begin{itemize}
\item[(1)] the supports of $\mathcal{D}$ and of the exceptional divisor $\mathcal{E}_{\mathcal{Z}/\mathcal{D}}$ are irreducible,
\item[(2)] $\mathcal{Z}\subset \mathcal{D}$ is an $\AA^{1}$-weak equivalence.
\end{itemize} 
Then there is a naturally induced $\AA^{1}$-weak equivalence between simplicial suspensions:
\begin{equation}
\label{equation:A1weakequivalencesuspension}
\Sigma_{s}\pi
\colon
\Sigma_{s}\widetilde{\mathcal{X}}(\mathcal{D},\mathcal{Z})
\overset{\sim_{\AA^{1}}}{\to}
\Sigma_{s}\mathcal{X}.
\end{equation} 

Furthermore, 
if $\widetilde{\mathcal{X}}(\mathcal{D},\mathcal{Z})$ is $\AA^{1}$-contractible,
then $\mathcal{X}$ is $\AA^{1}$-contractible.
\end{theorem}
\begin{proof}
Since $(\mathcal{Z}\subset \mathcal{D} \subset \mathcal{X})$ is a smooth triple, 
the affine modification $\tilde{\mathcal{X}}:=\tilde{\mathcal{X}}(\mathcal{D},\mathcal{Z})$ is isomorphic to the complement of the proper transform $\mathcal{D}'$ 
of $\mathcal{D}$ in the blow-up $\sigma:\Bl_{\mathcal{Z}}\mathcal{X}\to \mathcal{X}$ of $\mathcal{X}$ along $\mathcal{Z}$, 
and $\pi:\tilde{\mathcal{X}}\to \mathcal{X}$ coincides with the restriction of $\sigma$. 
The exceptional divisor $\mathcal{E}$ of $\sigma$ is isomorphic to the projective bundle $p:\mathbb{P}(\mathcal{N}_{\mathcal{Z}/\mathcal{X}})\to \mathcal{Z}$
of lines in the normal bundle $\mathcal{N}_{\mathcal{Z}/\mathcal{X}}$ of $\mathcal{Z}$ in $\mathcal{X}$. The normal
bundle $\mathcal{N}_{\mathcal{E}/\tilde{\mathcal{X}}}$ of $\mathcal{E}$ in $\Bl_{\mathcal{Z}}\mathcal{X}$ is equal to the
tautological line subbundle $\mathcal{O}_{\mathbb{P}(\mathcal{N}_{\mathcal{Z}/\mathcal{X}})}(-1)$
of $p^{*}\mathcal{N}_{\mathcal{Z}/\mathcal{X}}$ . From the exact sequence $0\to \mathcal{N}_{\mathcal{Z}/\mathcal{D}}\to \mathcal{N}_{\mathcal{Z}/\mathcal{X}}\to \mathcal{N}_{D/\mathcal{X}}|_{\mathcal{Z}}\to0$,
we obtain a closed embedding $\mathbb{P}(\mathcal{N}_{\mathcal{Z}/\mathcal{D}})\hookrightarrow\mathbb{P}(\mathcal{N}_{\mathcal{Z}/\mathcal{X}})$ 
as a hyperplane sub-bundle given by the zero locus of the global section of $\mathcal{O}_{\mathbb{P}(\mathcal{N}_{\mathcal{Z}/\mathcal{X}})}(1)\otimes p^{*}\mathcal{N}_{D/\mathcal{X}}|_{\mathcal{Z}}$
deduced from the composition $\mathcal{O}_{\mathbb{P}(\mathcal{N}_{\mathcal{Z}/\mathcal{X}})}(-1)\to p^{*}\mathcal{N}_{\mathcal{Z}/\mathcal{X}}\to p^{*}\mathcal{N}_{D/\mathcal{X}}|_{\mathcal{Z}}$. 

It follows from the setup that $\mathbb{P}(\mathcal{N}_{\mathcal{Z}/\mathcal{D}})\subset\mathbb{P}(\mathcal{N}_{\mathcal{Z}/\mathcal{X}})$ coincides with the intersection of the 
proper transform $\mathcal{D}'$ of $\mathcal{D}$ with $\mathcal{E}$. 
The restriction map $p_{\vert}:\mathcal{E}_{\mathcal{Z}/\mathcal{D}}=\mathcal{E}\setminus \mathcal{D}'=\mathbb{P}(\mathcal{N}_{\mathcal{Z}/\mathcal{X}})\setminus\mathbb{P}(\mathcal{N}_{\mathcal{Z}/\mathcal{D}})\to \mathcal{Z}$
is then a locally trivial $\mathbb{A}^{d}$-bundle, $d=\mathrm{codim}_{\mathcal{D}}(\mathcal{Z})$;
in particular, it is an $\mathbb{A}^{1}$-weak equivalence. 
Moreover,
since $\mathbb{P}(\mathcal{N}_{\mathcal{Z}/\mathcal{X}})$ is the zero locus of a section of $\mathcal{O}_{\mathbb{P}(\mathcal{N}_{\mathcal{Z}/\mathcal{X}})}(1)\otimes p^{*}\mathcal{N}_{D/\mathcal{X}}|_{\mathcal{Z}}$,
it follows that the restriction of $\mathcal{N}_{\mathcal{E}/\tilde{\mathcal{X}}}=\mathcal{O}_{\mathbb{P}(\mathcal{N}_{\mathcal{Z}/\mathcal{X}})}(-1)$ to $\mathcal{E}_{\mathcal{Z}/\mathcal{D}}$ coincides with 
$p^{*}\mathcal{N}_{D/\mathcal{X}}|_{\mathcal{Z}}$, 
hence it is equal to $p_{\vert}^{*}\mathcal{N}_{D/\mathcal{X}}|_{\mathcal{Z}}$. 

Now consider the following diagram of cofiber sequences: 
\begin{equation}
\label{equation:thomspacediagram}
\xymatrix{
\tilde{\mathcal{X}}\setminus \mathcal{E}_{\mathcal{Z}/\mathcal{D}} \ar[r]\ar[d]_{\pi_{\vert}} & 
\tilde{\mathcal{X}} \ar[r]\ar[d]_{\pi} &
\tilde{\mathcal{X}}/(\tilde{\mathcal{X}}\setminus \mathcal{E}_{\mathcal{Z}/\mathcal{D}})  \ar[d]_{\overline{\pi}}
\\ 
\mathcal{X}\setminus\mathcal{D} \ar[r] & 
\mathcal{X}\ar[r] &
\mathcal{X}/(\mathcal{X}\setminus \mathcal{D}). \\ 
}
\end{equation}
Since $\sigma^{-1}(\mathcal{D})=\mathcal{D}'\cup \mathcal{E}$, 
the leftmost vertical map $\pi_{\vert}:\tilde{\mathcal{X}}\setminus \mathcal{E}_{\mathcal{Z}/\mathcal{D}}\to \mathcal{X}\setminus \mathcal{D}$ is an isomorphism. 
By construction, 
the restriction of $\sigma$ to $\mathcal{E}_{\mathcal{Z}/\mathcal{D}}$ factors as the composition of 
$p_{\vert}:\mathcal{E}_{\mathcal{Z}/\mathcal{D}}=\mathcal{E}\setminus \mathcal{D}'\to \mathcal{Z}$ with $i:\mathcal{Z}\hookrightarrow \mathcal{D}$. 
On the other hand, 
by homotopy purity for closed embeddings \cite[Theorem 3.2.23]{MR1813224}, 
we have $\mathbb{A}^{1}$-weak equivalences: 
$$
\tilde{\mathcal{X}}/(\tilde{\mathcal{X}}\setminus \mathcal{E}_{\mathcal{Z}/\mathcal{D}})
\overset{\sim_{\AA^{1}}}{\to}
\mathrm{Th}(\mathcal{N}_{\mathcal{E}_{\mathcal{Z}/\mathcal{D}}/\tilde{\mathcal{X}}})
=
\mathrm{Th}(p_{\vert}^{*}\mathcal{N}_{\mathcal{D}/\mathcal{X}}|_{\mathcal{Z}})
=
\mathrm{Th}((i\circ p_{\vert})^{*}\mathcal{N}_{\mathcal{D}/\mathcal{X}})
\text{ and } 
\mathcal{X}/(\mathcal{X}\setminus \mathcal{D})
\overset{\sim_{\AA^{1}}}{\to}
\mathrm{Th}(\mathcal{N}_{\mathcal{D}/\mathcal{X}}). 
$$
Furthermore the map $\overline{\pi}:\tilde{\mathcal{X}}/(\tilde{\mathcal{X}}\setminus \mathcal{E}_{\mathcal{Z}/\mathcal{D}})\to \mathcal{X}/(\mathcal{X}\setminus \mathcal{D})$ 
induced by $\pi$ coincides via these equivalences with the one induced by the composition: 
$$
\mathcal{E}_{\mathcal{Z}/\mathcal{D}}
=
\mathcal{E}\setminus \mathcal{D}'
\stackrel{p_{\vert}}{\to}\mathcal{Z}
\stackrel{i}{\hookrightarrow}
\mathcal{D}.
$$
Here $p_{\vert}:\mathcal{E}_{\mathcal{Z}/\mathcal{D}}\to \mathcal{Z}$ is an $\mathbb{A}^{1}$-weak equivalence.
By (2) we conclude that $\overline{\pi}$ is an $\mathbb{A}^{1}$-weak equivalence.
Since $\Sigma_s(\widetilde{\mathcal{X}}\setminus \mathcal{E}_{\mathcal{Z}/\mathcal{D}})\to \Sigma_s(\mathcal{X}\setminus \mathcal{D})$ is an $\AA^{1}$-weak equivalence, 
it follows that $\Sigma_{s}\pi:\Sigma_s\widetilde{\mathcal{X}}\to \Sigma_s \mathcal{X}$ is an $\AA^{1}$-weak equivalence by simplicially 
suspending the homotopy cofiber sequences in \eqref{equation:thomspacediagram}.

Finally, 
assuming that $\widetilde{\mathcal{X}}$ is $\AA^{1}$-contractible, 
applying the very weak five lemma for pointed model categories \cite[Lemma 2.1]{MRDF} to \eqref{equation:thomspacediagram} 
implies $\mathcal{X}$ is $\AA^{1}$-contractible. 
\end{proof}

\begin{remark}
\label{remark:motivicirreduciblecomponents}
\aref{theorem:A1affinemodifications} holds more generally for finite decompositions into irreducible components $\mathcal{D}=\sum_{i=1}^{m} D_{i}$ and 
$\mathcal{E}_{\mathcal{Z}/\mathcal{D}}=\sum_{i=1}^{m} \mathcal{E}_{i}$ such that (1) and (2) holds componentwise.
For an analogue of \aref{theorem:A1affinemodifications} in the setting of motives we refer  to \cite[Proposition 4]{zbMATH05998029}.
\end{remark}

\begin{example}
\label{example:affinemodificationofR}
\begin{itemize}
\item
Let $\mathcal{R}$ be the Russell cubic defined in \eqref{equation:russell3fold}.
The affine modification of $\mathcal{R}$ along the divisor $\mathcal{D}:=\ddiv(x)$ with center $\mathcal{Z}$ given by the ideal $(x,t+1,y-1)$ is the hypersurface in $\AA^{4}$ given by the equation:
\begin{equation}
\label{equation:Raffinemodification}
\widetilde{\mathcal{X}}(\mathcal{D},\mathcal{Z})
:=
\{xz=x^{2}t^{3}-3xt+3t+y^{2}x+2y+1\}.
\end{equation}
\aref{theorem:miyanishi's characterization} shows $\widetilde{\mathcal{X}}(\mathcal{D},\mathcal{Z})$ is isomorphic to $\AA^{3}$, 
and hence $\AA^{1}$-contractible.
However, 
we cannot use \aref{theorem:A1affinemodifications} to conclude $\mathcal{R}$ is $\AA^{1}$-contractible because $(\mathcal{Z}\subset \mathcal{D}\subset \mathcal{R})$ is not a smooth triple.
\item
Let $\mathcal{X}$ be the hypersurface in $\AA^{6}=\Spec(F[x,y,z,t,u,v])$ defined by: 
$$
\{uv=x^{2}z-(y^{2}+t^{3}+x)\}. 
$$
Setting $\mathcal{D}:=\ddiv(u)\cong \mathcal{R}\times \AA^{1}_{v}$ and $\mathcal{Z}:=\{u=x=y=t=0\}\cong \AA^{2}_{v,z}$ the affine modification
$$
\widetilde{\mathcal{X}}(\mathcal{D},\mathcal{Z})
:=
\{v=x^{2}uz-y^{2}u-t^{3}u^{2}-x\}
\subset 
\AA^{6}
$$ 
is isomorphic to $\AA^{5}$. 
Since $\widetilde{\mathcal{X}}(\mathcal{D},\mathcal{Z})$, $\mathcal{D}$, and $\mathcal{Z}$ are $\AA^{1}$-contractible and the triple $(\mathcal{Z}\subset \mathcal{D}\subset \mathcal{X})$ is smooth, 
\aref{theorem:A1affinemodifications} implies $\mathcal{X}$ is $\AA^{1}$-contractible. 
Further arguments are required to determine whether or not $\mathcal{X}$ is isomorphic to $\AA^{5}$.
\item
We suggest a generalization of the previous example by constructing a potentially nontrivial family of $\AA^{1}$-contractible smooth affine schemes of higher dimensions. 
For $m\ge 5$ and pairwise coprime integers $\alpha_{i}$ we define hypersurfaces in $\AA^{m+1}=\Spec(F[u,v,x,z,y_1,\dots,y_{m-3}])$ by:
\begin{equation}
\mathcal{X}(n,\alpha_{1},\dots,\alpha_{m-3})
:=
\{uv=x^{n}z+y_1^{\alpha_1}+\dots+y_{m-3}^{\alpha_{m-3}}+x\}.
\end{equation}
Then $\mathcal{Z}=V(u,x,y_1,\dots,y_{m-1})\subset \mathcal{D}_u=\operatorname{div}(u)\subset \mathcal{X}(n,\alpha_i)$ is a smooth triple in $\Sm_k$ satisfying the assumptions 
in \aref{theorem:A1affinemodifications}, 
and:
$$
\widetilde{\mathcal{X}(n,\alpha_{1},\dots,\alpha_{m-3})}(\mathcal{D},\mathcal{Z})
\cong 
\AA^{m}
\sim_{\AA^{1}}
\ast.
$$ 
\end{itemize}
\end{example}

Recall from \aref{subsection:affinemodifications} the defomation space $\mathrm{D}(\mathcal{X},\mathcal{Z})$ of a closed embedding $\mathcal{Z}\hookrightarrow\mathcal{X}$ between smooth schemes.
The composite map
\[
\rho
=
\mathrm{pr}_{2}\circ\pi:
\mathrm{D}(\mathcal{X},\mathcal{Z})
=
\widetilde{\mathcal{X}\times\AA^{1}}(\mathcal{D},\mathcal{Z})
\to
\mathbb{A}^{1}
\]
is flat. 
Its restriction over $\mathbb{A}^{1}\setminus\{0\}$ is isomorphic to the projection on the second factor of $\mathcal{X}\times (\mathbb{A}^{1}\setminus\{0\})$, 
whereas its fiber over $\{0\}$ is canonically isomorphic to the normal bundle $\mathcal{N}_{\mathcal{Z}/\mathcal{X}}$ of $\mathcal{Z}$ in $\mathcal{X}$. 

\begin{corollary}
Let $i:\mathcal{Z}\hookrightarrow \mathcal{X}$ be a closed embedding between $\mathbb{A}^{1}$-contractible
irreducible smooth schemes. Then the simplicial suspension $\Sigma_{s}\mathrm{D}(\mathcal{X},\mathcal{Z})$
of $\mathrm{D}(\mathcal{X},\mathcal{Z})$ is $\mathbb{A}^{1}$-contractible. 
\end{corollary}
\begin{proof}
Since $\mathcal{Z}$ and $\mathcal{X}$ are both $\mathbb{A}^{1}$-contractible,
the inclusion $i\times\mathrm{id}:\mathcal{Z}=\mathcal{Z}\times\{0\}\hookrightarrow\mathcal{D}=\mathcal{X}\times\left\{ 0\right\}$ is an $\mathbb{A}^{1}$-weak equivalence. 
Theorem \ref{theorem:A1affinemodifications} implies $\Sigma_{s}\pi:\Sigma_{s}\mathrm{D}(\mathcal{X},\mathcal{Z})\to\Sigma_{s}\mathcal{X}$ is an $\mathbb{A}^{1}$-weak equivalence.
We are done since $\mathcal{X}\times\mathbb{A}^{1}$ is $\mathbb{A}^{1}$-contractible. 
\end{proof}

\begin{example}
As examples of deformation spaces we look at the origin, a special line and a particular affine plane in the Koras-Russell $3$-fold $\mathcal{X}=\mathcal{X}(n,\alpha_{i},1):=\{x^{n}z=y^{\alpha_{1}}+t^{\alpha_{2}}+x\}\subset\AA^{4}$ \eqref{equation:firstkind}.
As a matter of notation we write $\mathbb{A}^{4}\times\mathbb{A}^{1}=\mathrm{Spec}(F[x,y,z,t][u])$.
\begin{itemize}
\item $\mathcal{Z}=\left\{(0,0,0,0)\right\} $. 
The ideal of $\mathcal{Z}\times\left\{0\right\}$ in $\mathcal{X}\times\AA^{1}$ is globally generated by the regular sequence $u,y,z,t$. 
The deformation space $\mathrm{D}(\mathcal{X},\mathcal{Z})$ is thus isomorphic to the smooth $4$-fold 
\[
\{x^{n}uz_{1}=(uy_{1})^{\alpha_{1}}+(ut_{1})^{\alpha_{2}}+x\}
\subset
\mathbb{A}^{4}\times\mathbb{A}^{1}=\mathrm{Spec}(F[x,y_{1},z_{1},t_{1}][u]),
\]
and the modification map $\pi:\mathrm{D}(\mathcal{X},\mathcal{Z})\to\mathcal{X}\times\AA^{1}$ is given by $(x,y_{1},z_{1},t_{1},u)\mapsto(x,uy_{1},uz_{1},ut_{1},u)$. 
\item
$\mathcal{Z}=\{x=y=t=0\}$. 
The ideal of $\mathcal{Z}\times\left\{0\right\}$ in $\mathcal{X}\times\AA^{1}$ is globally generated by $u,x,y,t$.
We find that the deformation space $\mathrm{D}(\mathcal{X},\mathcal{Z})$ is isomorphic to the smooth $4$-fold 
\[
\{x_{1}^{n}u^{n-1}z=u^{\alpha_{1}-1}y_{1}{}^{\alpha_{1}}+u^{\alpha_{2}-1}t_{1}{}^{\alpha_{2}}+x_{1}\}
\subset
\mathbb{A}^{4}\times\mathbb{A}^{1}=\mathrm{Spec}(F[x_{1},y_{1},z,t_{1}][u]),
\]
with the modification map $\pi:\mathrm{D}(\mathcal{X},\mathcal{Z})\to\mathcal{X}\times\AA^{1}$ given by $(x_{1},y_{1},z,t_{1},u)\mapsto(ux_{1},uy_{1},z,ut_{1},u)$. 
\item
$\mathcal{Z}=\{z=0\}$. 
The ideal of $\mathcal{Z}\times\left\{0\right\}$ in $\mathcal{X}\times\AA^{1}$ is globally generated by the regular sequence $u,z$. 
The deformation space $\mathrm{D}(\mathcal{X},\mathcal{Z})$ is isomorphic to the smooth $4$-fold
\[
\{x^{n}uz_{1}=y^{\alpha_{1}}+t^{\alpha_{2}}+x\}
\subset
\mathbb{A}^{4}\times\mathbb{A}^{1}=\mathrm{Spec}(F[x,y,z_{1},t][u]),
\]
and the modification map $\pi:\mathrm{D}(\mathcal{X},\mathcal{Z})\to\mathcal{X}\times\AA^{1}$ is given by $(x,y,z_{1},t,u)\mapsto(x,y,uz_{1},t,u)$. \\
\end{itemize}
\vspace{-0.1in}
In each case, 
$\mathrm{D}(\mathcal{X},\mathcal{Z})$ is smooth and the flat map $\rho:\mathrm{D}(\mathcal{X},\mathcal{Z})\to\mathbb{A}^{1}$ coincides with the restriction of 
$\mathrm{pr}_{u}\colon\mathbb{A}^{4}\times\mathbb{A}^{1}\to\mathbb{A}^{1}$. 
Here $\rho$ restricts to a trivial bundle $\mathcal{X} \times (\mathbb{A}^{1}\setminus\{0\})$ over $(\mathbb{A}^{1}\setminus\{0\})$ and $\rho^{-1}(0)$ is isomorphic to $\mathbb{A}^{3}$. 
Furthermore, 
since $\mathcal{X}$ is $\mathbb{A}^{1}$-contractible,
so is the simplicial suspension $\Sigma_{s}\mathrm{D}(\mathcal{X},\mathcal{Z})$.
We do not know in any of these cases whether $\mathrm{D}(\mathcal{X},\mathcal{Z})$ is $\mathbb{A}^{1}$-contractible. 
When the base field $F$ is $\mathbb{C}$, 
it follows from \aref{theorem:Kaliman-Zaidenbergcriterion} that $\mathrm{D}(\mathcal{X},\mathcal{Z})$ is contractible. 
\end{example}

\subsection{$\AA^{1}$-chain connectedness}
\label{subsection: A1-chainconnected}
We say that an $F$-scheme $\mathcal{X}\in\Sch_{F}$ is $\AA^{1}$-chain connected if $\mathcal{X}(E)\neq\emptyset$ for every finitely generated separable field extension $E/F$, 
and for all rational points $x,y\in \mathcal{X}(E)$ there exist rational points $x=x_{0},x_{1},\dots,x_{n-1},x_{n}=y\in \mathcal{X}(E)$ and an elementary $\AA^{1}$-equivalence 
$f_{i}:\AA^{1}_{E}\to \mathcal{X}$ between $x_{i-1}$ and $x_{i}$ for $1\leq i\leq n$.
Equivalently, 
$\pi_{0}(\SSing^{\AA^{1}}(\mathcal{X})(E))=0$ for the $\AA^{1}$-singular functor of \cite{MR1813224} and every $E/F$ as above.

\begin{remark}
Our definition of $\AA^{1}$-chain connectedness is that of \cite[Definition 2.2.2]{AsokOttawa}, \cite[Definition 2.2.2]{MR2803793} generalized to $\Sch_{F}$.
All $\AA^{1}$-chain connected smooth schemes are $\AA^{1}$-connected \cite[Proposition 2.2.7]{MR2803793}. 
The converse implication is false by \cite[\S4]{zbMATH06473189}.
\end{remark}

\begin{question}
\label{question:A1contractibleimpliesA1chainconnected}
Is every $\AA^{1}$-contractible smooth scheme also $\AA^{1}$-chain connected? 
\end{question}

We give some supporting examples for \aref{question:A1contractibleimpliesA1chainconnected} based on the affine modifications in \aref{table:1}. 
Antieau \cite[Proposition 2.4.6]{AsokOttawa} observed the following result for the Russell cubic $\mathcal{R}$ in \eqref{equation:russell3fold}.
\begin{example}
\label{example:RussellcubicA1-chainconnected}
Every Koras-Russell $3$-fold of the first kind $\mathcal{X}(n,\alpha_{i},1)$ in \eqref{equation:firstkind} is $\AA^{1}$-chain connected.
\end{example}
\begin{proof}
Let $E/F$ be a finitely generated separable field extension.
For rational points $0,(y_{0},t_{0})\in \Gamma_{\alpha_{1},\alpha_{2}}(E)$ on the cuspidal curve $\{y^{\alpha_{1}}+t^{\alpha_{2}}=0\}$ there is an elementary $\AA^{1}$-equivalence 
$\AA^{1}\to \Gamma_{\alpha_{1},\alpha_{2}}$ given by: 
$$
s\mapsto (y_{0}s^{\alpha_{2}},t_{0}s^{\alpha_{1}}).
$$
Thus the fibers $\AA^{2}_{y,t}$ and $\Gamma_{\alpha_{1},\alpha_{2}}\times\AA^{1}_{z}$ of the projection $\pr_{x}:\mathcal{X}(n,\alpha_{i},1)\to \AA^{1}_{x}$ are $\AA^{1}$-chain connected.

To construct elementary $\AA^{1}$-equivalences between points in different fibers we choose polynomials $y(s),t(s)\in F[s]$ such that $s^{n}$ divides $y(s)^{\alpha_{1}}+t(s)^{\alpha_2}+s$. 
For a rational point $a\in \AA^{1}(E)$ the map $\AA^{1}\to \mathcal{X}(n,\alpha_{i},1)$ defined by 
$$
s
\mapsto 
(x(s),z(s),y(s),t(s))=(as, \frac{y(as)^{\alpha_{1}}+t(as)^{\alpha_2}+as}{(as)^{n}}, y(as),t(as))
$$ 
connects a point in the $0$-fiber with a point in the $a\neq 0$-fiber of $\operatorname{pr}_{x}$.
\end{proof}
\begin{remark}
For the Russell cubic $\mathcal{R}$ we can choose the elementary $\AA^{1}$-equivalence $\AA^{1}\to \mathcal{R}$ defined by: 
$$
s
\mapsto 
(x(s),z(s),y(s),t(s))=(as,-2-as,as+1,-as-1).
$$
\end{remark}

\begin{example}
\label{example:BPA1chain}
Every Brieskorn-Pham surface $\mathcal{S}_{\alpha_{i}}$ in \eqref{equation:BPsurface} is $\AA^{1}$-chain connected.
\end{example}
\begin{proof}
For $a_{i}\in \mathcal{S}_{\alpha_{i}}(E)$ and $E/F$ as above,
the map $\AA^{1}\to \mathcal{S}_{\alpha_{i}}$ defined by 
$$
s
\mapsto 
(a_{1}s^{\alpha_{2}\alpha_{3}}, a_{2}s^{\alpha_{1}\alpha_{3}},a_{3}s^{\alpha_{1}\alpha_{2}})
$$ 
is an elementary $\AA^{1}$-equivalence between the rational points $0$, $a_{i}\in \mathcal{S}_{\alpha_{i}}(E)$.
\end{proof}

\begin{example}
The $4$-fold $\mathcal{Y}(2,(2,3,5),1):=\{x^{2}z=y_{1}^{2}+y_2^{3}+y_3^{5}+x\}\subset\AA^{5}$ in \eqref{equation:4dimensional} is $\AA^{1}$-chain connected. 
The fibers $\AA^{3}$ and $\mathcal{S}_{2,3,5}$ of the projection map $\pr_{x}:\mathcal{Y}(2,(2,3,5),1)\to \AA^{1}_{x}$ are $\AA^{1}$-chain connected (see \aref{example:BPA1chain}).
An elementary $\AA^{1}$-equivalence between the $0$-fiber and the $a\neq 0$-fiber is given by: 
$$
s
\mapsto 
(x(s),z(s),y_{1}(s),y_2(s),y_3(s))=(as,-2-as,as+1,-as-1,0).
$$
\end{example}

\begin{example}
\label{example:tomDieck-Petriesurfaces}
For coprime integers $k>l\geq 2$, 
the smooth tom Dieck-Petrie surface 
\begin{equation}
\label{equation:tomDieck-Petriesurfaces}
\mathcal{V}_{k,l}
:=
\{\frac{(xz+1)^k-(yz+1)^l-z}{z}=0\}
\subset 
\AA^{2}
\end{equation}
is the affine modification of $\mathcal{Z}=(1,1)\subset \mathcal{D}=\{x^k-y^l=0\} \subset \mathcal{X}=\AA^{2}$.
It is topologically contractible by \aref{theorem:Kaliman-Zaidenbergcriterion}, 
and stably $\AA^{1}$-contractible by \aref{theorem:stableA1affinemodification} and \cite[Proposition 3.2]{MR3549169}. 
However, $\mathcal{V}_{k,l}$ has logarithmic Kodaira dimension $\overline{\kappa}(\mathcal{V}_{k,l})=1$, 
and thus it is not $\AA^{1}$-chain connected (see \aref{subsection:a1contractiblesmoothaffinesurfaces}).
This example shows that the affine modification construction does not preserve $\AA^{1}$-chain connectedness.
\end{example}

\subsection{$\AA^{1}$-contractible smooth affine surfaces}
\label{subsection:a1contractiblesmoothaffinesurfaces}
Let $F$ be an algebraically closed field of characteristic $0$.
$\AA^{1}$-contractible smooth affine curves over $F$ are isomorphic to the affine line $\AA^{1}$ by \cite[Claim 5.7]{MR2335246}.
Let $\mathcal{S}$ be an $\AA^{1}$-contractible smooth affine surface over $F$. 
By \cite[Lemma 2.3.8]{AsokOttawa} $\AA^{1}$-contractibility implies topological contractibility.
From \cite{MR687591} it follows that $\mathcal{S}$ has logarithmic Kodaira dimension $\overline{\kappa}(\mathcal{S})$ equal to $-\infty$, $1$, or $2$. 
In case $\overline{\kappa}(\mathcal{S})=-\infty$, 
the birational classification results due to Miyanishi, Sugie and Fujita shows $\mathcal{S}$ contains a nonempty open subset isomorphic to $\AA^{1}\times\mathcal{C}$, 
where $\mathcal{C}$ is a smooth affine curve \cite{MR564667}. 
By Miyanishi's algebraic characterization of the affine plane we conclude $\mathcal{S}$ is isomorphic to $\AA^{2}$ \cite{MR0419460}. 
This reduces the question whether $\AA^{2}$ is the only smooth affine $\AA^{1}$-contractible surface to the following.

\begin{question}
\label{question:A1contractiblenegativelogkodaira}
Does every $\AA^{1}$-contractible smooth surface over $F$ have negative logarithmic Kodaira dimension?
(An affirmative answer to \aref{question:A1contractibleimpliesA1chainconnected} would show this is the case.)
\end{question}

\begin{remark}
The Zariski cancellation problem for $\mathcal{X}$ a smooth affine $d$-dimensional scheme asks whether $\mathcal{X}\times\AA^{1} \cong \AA^{d+1}$ implies $\mathcal{X}\cong \AA^{d}$.
This has been settled affirmatively for curves by Abhyankar, Eakin, and Heinzer \cite{AEH}, and for surfaces by Fujita \cite{Fujita}, Miyanishi and Sugie \cite{MR564667}.
When $d\geq 3$ this remains an open problem over fields of characteristic zero, 
while there exists counterexamples over fields of positive characteristic;
we refer to \cite{MR3438041} for a survey.
Note that if $\mathcal{X}$ is stably isomorphic to an affine space then it is $\AA^{1}$-contractible, 
e.g., 
the Russell cubic $\mathcal{R}$ is a potential counterexample to Zariski cancellation.
\end{remark}

All $\AA^{1}$-chain connected smooth schemes in the sense of Asok-Morel \cite{MR2803793} are log-uniruled and hence of negative logarithmic Kodaira dimension \cite{MR3351185}. 
This implies the following classification result.
\begin{prop}
\label{prop:isotoaffineplane}
An $\AA^{1}$-contractible and $\AA^{1}$-chain connected smooth affine surface over an algebraically closed field of characteristic $0$ is isomorphic to the affine plane $\AA^{2}$.
\end{prop}

A landmark result due to Ramanujam \cite{MR0286801} tells us that any smooth complex surface whose underlying analytic space is contractible and simply connected at infinity is 
isomorphic to $\AA^{2}$.
\footnote{An open manifold $\mathcal{M}$ is said to be simply connected at infinity if for every compact subset $\mathcal{C}\subset\mathcal{M}$,
there is a compact subset $\mathcal{C}'$ such that $\mathcal{C}\subset\mathcal{C}'\subset\mathcal{M}$ and $\mathcal{M}\setminus\mathcal{C}'$ is connected and simply connected.}
\begin{question}
Are $\AA^{1}$-contractible smooth affine surfaces simply connected at infinity?
(An affirmative answer to \aref{question:A1contractiblenegativelogkodaira} would show this is the case.)
\end{question}

\begin{remark}
There exist non $\AA^{1}$-chain connected topologically contractible smooth affine surfaces non-isomorphic to $\AA^{2}$,
e.g., 
the Ramanujam surface \cite{MR0286801} and the tom Dieck-Petrie surfaces (see \aref{example:tomDieck-Petriesurfaces}).
Proposition 4.1 in \cite{MR2423761} shows that all vector bundles on these surfaces are trivial,
i.e., 
algebraic $K$-groups cannot disprove $\AA^{1}$-contractibility of these surfaces.
\end{remark}

\section{Higher Chow groups calculations and stable $\AA^{1}$-contractibility}
\label{section:hcgc}
In the following we will make use of a criterion for stable $\AA^{1}$-contractibility shown in \cite[$\S4$]{MR3549169}.

\begin{theorem}
\label{theorem:higherChowGroups}
Suppose $\mathcal{X}$ is a smooth scheme equipped with a rational point $x\in\mathcal{X}(F)$.
If the natural map $\mathcal{X}\times\mathcal{Y}\to\mathcal{Y}$ induces an isomorphism on higher Chow groups for any smooth affine scheme $\mathcal{Y}$, 
then $\mathcal{X}$ is stably $\AA^{1}$-contractible in the sense that it becomes $\AA^{1}$-contractible after a finite suspension with the projective line pointed at infinity:
\[
(\mathcal{X},x)\wedge \PP^{1}\wedge \cdots \wedge \PP^{1}\sim_{\AA^{1}}\ast.
\]
\end{theorem}

\begin{theorem}
\label{theorem:stableA1affinemodification}
Let $\mathcal{X}\in\Sch_{F}$ be an integral scheme. 
Suppose $\mathcal{D}$ is an effective Cartier divisor on $\mathcal{X}$, 
and $\mathcal{Z}$ a closed subscheme of $\mathcal{X}$ with ideal sheaf $\II_{\mathcal{Z}}\subset\mathscr{O}_{\mathcal{X}}(-\mathcal{D})$ which is a local complete intersection in $\mathcal{D}$.
\begin{enumerate}
\item[(i)] If for every smooth affine scheme $\mathcal{Y}$, 
the maps $\pr_{\mathcal{Y}}:\mathcal{Z}\times \mathcal{Y}\to \mathcal{Y}$, $\pr_{\mathcal{Y}}:\mathcal{D}\times \mathcal{Y}\to \mathcal{Y}$, 
and $(f,\pr_{\mathcal{Y}}):\mathcal{X}\times \mathcal{Y}\to \AA^{1}\times \mathcal{Y}$ induce isomorphisms on higher Chow groups, 
then so does $\widetilde{\mathcal{X}}(\mathcal{D},\mathcal{Z})\times \mathcal{Y}\to \mathcal{Y}$. 
In addition,
if the affine modification $\widetilde{\mathcal{X}}(\mathcal{D},\mathcal{Z})$ is smooth, 
then it is stably $\AA^{1}$-contractible.
\item[(ii)] If for every smooth affine scheme $\mathcal{Y}$, 
the maps $\pr_{\mathcal{Y}}:\mathcal{Z}\times \mathcal{Y}\to \mathcal{Y}$, $\pr_{\mathcal{Y}}:\mathcal{D}\times \mathcal{Y}\to \mathcal{Y}$, 
and $(f,\pr_{\mathcal{Y}}):\widetilde{\mathcal{X}}(\mathcal{D},\mathcal{Z})\times \mathcal{Y}\to \AA^{1}\times \mathcal{Y}$ induce isomorphisms on higher Chow groups, 
then so does $\mathcal{X}\times \mathcal{Y}\to \mathcal{Y}$. 
In addition, 
if $\mathcal{X}$ is smooth, 
then it is stably $\AA^{1}$-contractible.
\end{enumerate}
\end{theorem}
\begin{proof}
Let $\mathcal{V}\mathcal{W}$ be shorthand for the fiber product $\mathcal{V}\times\mathcal{W}$ in $\Sch_{F}$, 
let $\widetilde{\mathcal{X}}$ be shorthand for $\widetilde{\mathcal{X}}(\mathcal{D},\mathcal{Z})$, 
and set $\mathcal{E}:=\mathcal{E}_{\mathcal{Z}/\mathcal{D}}$.
To prove $\text{(i)}$, 
note that $\mathcal{E}$ and $\mathcal{U}=\widetilde{\mathcal{X}}\setminus \mathcal{E} $ give rise to a commutative diagram of exact localization sequences for higher Chow groups: 
\begin{equation}
\label{equation:chowgroupaffinemodification}
\xymatrix{
\CH_*(\GG_m \mathcal{Y},i+1)  \ar[r] \ar[d]_{(1)} & 
\CH_*(\mathcal{Y},i) \ar[r]\ar[d]_{(2)} &
\CH_*(\AA^{1} \mathcal{Y},i)  \ar[r]\ar[d]_{(3)} &  
\CH_*(\GG_m \mathcal{Y},i)  \ar[r]\ar[d]_{(4)} & 
\CH_*(\mathcal{Y},i-1)\ar[d]_{(5)}\\ 
\CH_*(\mathcal{U}\mathcal{Y},i+1)  \ar[r] & 
\CH_*(\mathcal{E} \mathcal{Y},i) \ar[r] &
\CH_*(\widetilde{\mathcal{X}} \mathcal{Y},i)  \ar[r]&  
\CH_*(\mathcal{U}\mathcal{Y},i)  \ar[r] & 
\CH_*(\mathcal{E} \mathcal{Y},i-1).  \\ 
}
\end{equation}
Here $\pi_{\mathcal{D}/\mathcal{Z}}\vert_{\mathcal{E}}:\mathcal{E}\to\mathcal{Z}$ is a locally trivial $\AA^d$-bundle, 
$d=\operatorname{codim}_{\mathcal{D}}(\mathcal{Z})$, 
by the assumption on $\mathcal{Z}$.
Hence, the map $\pr_{\mathcal{Y}}:\mathcal{E}\to \mathcal{Y}$ induces an isomorphism on higher Chow groups by hypothesis and homotopy invariance.
It follows that $\text{(2)}$ and $\text{(5)}$ in \eqref{equation:chowgroupaffinemodification} are isomorphisms.
Furthermore, 
$\mathcal{U}\cong \mathcal{X}\setminus \mathcal{D}$ and since the maps $\pr_{\mathcal{Y}}:\mathcal{D}\mathcal{Y}\to \mathcal{Y}$ and 
$(f,\pr_{\mathcal{Y}}):\mathcal{X}\mathcal{Y}\to \AA^{1}\mathcal{Y}$ induce isomorphisms on higher Chow groups, 
applying the five lemma to \eqref{equation:chowgroupaffinemodification2} shows $\text{(1)}$ and $\text{(4)}$ in \eqref{equation:chowgroupaffinemodification} are isomorphisms:
\begin{equation}
\label{equation:chowgroupaffinemodification2}
\xymatrix{
\CH_*(\mathcal{Y},i)  \ar[r]\ar[d] & 
\CH_*(\AA^{1} \mathcal{Y},i) \ar[r]\ar[d] &
\CH_*(\GG_m \mathcal{Y},i)  \ar[r]\ar[d]&  
\CH_*(\mathcal{Y},i-1)  \ar[r]\ar[d] & 
\CH_*(\AA^{1} \mathcal{Y},i-1)\ar[d]\\ 
\CH_*(\mathcal{D} \mathcal{Y},i)  \ar[r] & 
\CH_*(\mathcal{X} \mathcal{Y},i) \ar[r] &
\CH_*(\mathcal{U}\mathcal{Y},i)  \ar[r]&  
\CH_*(\mathcal{D} \mathcal{Y},i-1)  \ar[r] & 
\CH_*(\mathcal{X} \mathcal{Y},i-1). \\ 
}
\end{equation}
It follows that $\text{(3)}$ in \eqref{equation:chowgroupaffinemodification} is an isomorphism, 
and we conclude by homotopy invariance.
\aref{theorem:higherChowGroups} implies the assertion of stable $\AA^{1}$-contractibility.

Part $\text{(ii)}$ is proved similarly by replacing $\widetilde{\mathcal{X}}$ with $\mathcal{X}$ and interchanging the roles of $\mathcal{E}$ and $\mathcal{D}$.
\end{proof}

\begin{corollary}
\label{corollary:deformed KR stable}
The deformed Koras-Russel $3$-folds of the first kind in \eqref{equation:firstkindpolynomial} are stably $\AA^{1}$-contractible.
\end{corollary}
\begin{proof}
Recall that $\mathcal{X}(n,\alpha_i,p)$ in \eqref{equation:firstkindpolynomial} is the affine modification of $\AA^{3}$ along $\mathcal{D}:=\operatorname{div}(x^{n})$ with center $\mathcal{Z}$ 
given by the ideal $(x^{n},y^{\alpha_1}+t^{\alpha_2}+xp(x,y,t))$, 
cf. \aref{table:1}.
We identify the reduced scheme $\mathcal{Z}_{\operatorname{red}}$ with the cuspidal curve $\Gamma_{\alpha_1,\alpha_2}=\{y^{\alpha_1}+t^{\alpha_2}=0\}\subset \AA^{2}$, 
so that $\CH_*(\mathcal{Z},i)\cong\CH_*(\Gamma_{\alpha_1,\alpha_2},i)$ (see \cite{MR1269719}).
If $\mathcal{Y}$ is a smooth affine scheme,
then $\mathcal{Z}\times \mathcal{Y}\to \mathcal{Y}$ induces an isomorphism in higher Chow groups \cite[Proposition 3.2]{MR3549169}.
The same conclusion holds for the maps $\mathcal{D}\times \mathcal{Y}\to \mathcal{Y}$ and 
$(\pr_{x},\pr_\mathcal{Y}):\mathcal{X}\times \mathcal{Y}=\AA^{3}\times \mathcal{Y}\to \AA^{1}\times \mathcal{Y}$ by homotopy invariance.
Applying \aref{theorem:stableA1affinemodification} finishes the proof.
\end{proof}

\begin{corollary}
\label{corollary:iterated stable} 
The iterated Koras-Russell $3$-folds in \eqref{equation:iteratedfirstkindpolynomial} are stably $\AA^{1}$-contractible.
\end{corollary}
\begin{proof}
Recall that $\mathcal{Y}(m,n_i,\alpha_i,p)$ in \eqref{equation:iteratedfirstkindpolynomial} is the affine modification 
$\widetilde{\mathcal{X}(n,\alpha_i,p)}(\mathcal{D}=\operatorname{div}(x^{n_2}), \mathcal{Z}=V(\mathcal{I}))$ where $\mathcal{I}=(x^{n_2}, y-z^{m})$, 
cf. \aref{table:1}.
Here $\mathcal{Z}_{\operatorname{red}}$ is isomorphic $\Gamma_{m\alpha_1,\alpha_2}$ and $\mathcal{D}_{\operatorname{red}}$ is isomorphic to the cylinder on $\Gamma_{\alpha_1,\alpha_2}$. 
For any smooth affine scheme $\mathcal{Y}$, 
\cite[Proposition 3.2]{MR3549169} shows the maps $\mathcal{D}\times \mathcal{Y}\to \mathcal{Y}$ and $\mathcal{Z}\times \mathcal{Y}\to \mathcal{Y}$ 
induce isomorphisms on higher Chow groups.
The same holds for $(\pr_{x},\pr_\mathcal{Y}):\mathcal{X}(n,\alpha_i,p)\times \mathcal{Y}\to \AA^{1}\times \mathcal{Y}$ by \aref{corollary:deformed KR stable}. 
To conclude we apply \aref{theorem:stableA1affinemodification}.
\end{proof}

We show stable $\AA^{1}$-contractibility of $4$-folds introduced by Kaliman-Zaidenberg in \cite{MR1817379}.
These can be distinguished from $\AA^{4}$ by the Makar-Limanov invariant.
\begin{corollary}
For $m\ge 2$ and coprime integers $k>l\geq 3$, 
the $4$-fold defined by
\begin{equation}
\label{equation:KalimanZaidenbergExample}
\{u^mv=\frac{(xz+1)^k-(yz+1)^l-z}{z}\}
\subset
\AA^{5}
\end{equation}
is stably $\AA^{1}$-contractible.
\end{corollary}
\begin{proof}
The $4$-fold in \eqref{equation:KalimanZaidenbergExample} is the affine modification of $\AA^{4}_{u,x,y,z}$ along $\operatorname{div}(u^{2})$ with center $\mathcal{Z}$ given by: 
$$
(u^m,\frac{(xz+1)^k-(yz+1)^l-z}{z}).
$$ 
Here $\mathcal{Z}$ is isomorphic to a tom Dieck-Petrie surface, 
i.e., 
an affine modification of $\AA^{2}$ (see \aref{example:tomDieck-Petriesurfaces}). 
Applying \aref{theorem:stableA1affinemodification} to these affine modifications yields the claim.
\end{proof}
\begin{remark}
Setting $m=2$, $k=4$, and $l=2$ the hypersurface in \eqref{equation:KalimanZaidenbergExample} is given by: 
\begin{equation}
\label{equation:KalimanZaidenbergExample2}
\{u^2v=x^4z^3+4x^3z^2+6x^2z+4x-y^3z^2-3y^2z-3y-1\}.
\end{equation}
Its affine modification along $\operatorname{div}(u^2)$ with center given by $I=(u^2,x-1,y-1,z)$ is isomorphic to $\AA^4$. 
If the tom Dieck-Petrie surfaces are $\AA^1$-contractible, 
then \eqref{equation:KalimanZaidenbergExample2} is $\AA^1$-contractible by \aref{theorem:A1affinemodifications}.
\end{remark}

We give two more examples of stably $\AA^{1}$-contractible smooth affine $4$-folds starting with \aref{thm:4dimensional-A1Cont}.
\begin{corollary}
\label{corollary:c=mab+1}
If $\alpha_{3}=m\alpha_{1}\alpha_{2}+1$, 
$m>0$, 
then $\mathcal{Y}(n,\alpha_i,p)$ in \eqref{equation:4dimensional} is stably $\AA^{1}$-contractible.
\end{corollary}
\begin{proof}
Recall $\mathcal{Y}:=\mathcal{Y}(n,\alpha_i,p)$ is the affine modification of $\AA^{4}$ along $\ddiv(x^{n})$ and with reduced center isomorphic to the surface 
$\mathcal{S}=\{y_1^{\alpha_1}+y_2^{\alpha_2}+y_3^{\alpha_3}=0\}\subset \AA^{3}$, 
cf. \aref{table:1}.
By \aref{theorem:stableA1affinemodification} it suffices to show $\mathcal{S}\times \mathcal{Y}\to \mathcal{Y}$ induces an isomorphism on higher Chow groups for 
any smooth affine scheme $\mathcal{Y}$. 
The map $(\pr_{y_{3}},\id_\mathcal{Y}):\mathcal{S}\times \mathcal{Y}\to \AA^{1}\times \mathcal{Y}$ fits into the commutative diagram of exact localization sequences: 
\begin{equation}
\label{equation:surface c=ab+1}
\xymatrix{
\CH_*(\GG_m \mathcal{Y},i+1)  \ar[r]\ar[d]_{(1)} & 
\CH_*(\mathcal{Y},i) \ar[r]\ar[d]_{(2)} &
\CH_*(\AA^{1} \mathcal{Y},i)  \ar[r]\ar[d]_{(3)} &  
\CH_*(\GG_m \mathcal{Y},i)  \ar[r]\ar[d]_{(4)} & 
\CH_*(\mathcal{Y},i-1)\ar[d]_{(5)} \\ 
\CH_*(\mathcal{U},i+1)  \ar[r] & 
\CH_*(\Gamma\mathcal{Y},i) \ar[r] &
\CH_*(\mathcal{S} \mathcal{Y},i)  \ar[r]&  
\CH_*(\mathcal{U},i)  \ar[r] & 
\CH_*(\Gamma\mathcal{Y},i-1).  \\ 
}
\end{equation}
Here $\Gamma:=\Gamma_{\alpha_{1},\alpha_{2}}$ and $\mathcal{U}:=(\mathcal{S}\times \mathcal{Y})\setminus \{y_3=0\}$. 
By \cite[Proposition 3.2]{MR3549169} the maps $\text{(2)}$ and $\text{(5)}$ in \eqref{equation:surface c=ab+1} are isomorphisms. 
We note there is an isomorphism $\mathcal{U}\cong \AA^{2}\setminus \Gamma_{\alpha_1,\alpha_2}=\Spec(F[u,v])\setminus \Gamma_{\alpha_1,\alpha_2}$ via the map 
$$
(x,y,z)
\mapsto 
(u(-u^{\alpha_1}-v^{\alpha_2})^{m\alpha_2},v(-u^{\alpha_1}-v^{\alpha_2})^{m\alpha_1},-u^{\alpha_1}-v^{\alpha_2}),
$$ 
with inverse: 
$$
(u,v)
\mapsto 
(\frac{x}{z^{m\alpha_2}},\frac{y}{z^{m\alpha_1}}).
$$ 
The five lemma and \cite[Proposition 3.2]{MR3549169} show $\text{(1)}$ and $\text{(4)}$ in \eqref{equation:surface c=ab+1} are isomorphisms, 
cf. the proof of \aref{theorem:stableA1affinemodification}.
Another application of the five lemma yields that $\text{(3)}$ in \eqref{equation:surface c=ab+1} is also an isomorphism.
\end{proof}

\begin{corollary}
The smooth affine $4$-fold $\mathcal{Y}(n,(2,3,5),p)$ in \eqref{equation:4dimensional} is stably $\AA^{1}$-contractible.
\end{corollary}
\begin{proof}
As in \aref{corollary:c=mab+1} it suffices to check that $\mathcal{S}_{1}\times \mathcal{Y}\to \mathcal{Y}$ induces an isomorphism on higher Chow groups, 
where $\mathcal{Y}$ is an arbitrary smooth affine scheme and $\mathcal{S}_{1}:=\{y_1^{2}+y_2^{3}+y_3^{5}=0\}\subset\AA^{3}$.
In effect,
we form the affine modifications:
\begin{equation}
\mathcal{S}_4\xrightarrow{\pi_3}\mathcal{S}_3\xrightarrow{\pi_2}\mathcal{S}_2\xrightarrow{\pi_1}\mathcal{S}_1.
\end{equation}
Here $\mathcal{S}_2:=\{y_1^{2}+y_2^{3}y_3+y_3^{3}\}\subset \AA^{3}$ is the affine modification of $\mathcal{S}_1$ along $\operatorname{div}(y_3)$ with center given by $(y_3,y_1,y_2)$, 
$\mathcal{S}_3:=\{y_1^{2}+y_2^2y_3+y_3^3y_2=0\}\subset \AA^{3}$ is the affine modification of $\mathcal{S}_2$ along $\operatorname{div}(y_2)$ with center given by $(y_2,y_1,y_3)$, 
and $\mathcal{S}_4:=\{y_1^{2}+y_2^{2}y_3+y_2=0\}\subset \AA^{3}$ is the affine modification of $\mathcal{S}_3$ along $\operatorname{div}(y_3^{3})$ with center given by $(y_3^{3},y_1,y_2)$.
By \aref{theorem:stableA1affinemodification}, 
$\mathcal{S}_1\times \mathcal{Y}\to \mathcal{Y}$ induces an isomorphism on higher Chow groups if and only if the same holds for $\mathcal{S}_4\times \mathcal{Y}\to \mathcal{Y}$.
Now $\mathcal{S}_4$ is the affine modification of $\AA^{2}$ along $\operatorname{div}(y_2^{2})$ with center given by $(y_2^{2},-y_1^{2}-y_2)$, 
and we are done by \aref{theorem:stableA1affinemodification}.
\end{proof}

\section{Unstable $\AA^{1}$-contractibility}
\label{section:unstable}
\subsection{Koras-Russell fiber bundles}
For the convenience of the reader we recall the notion of a Koras-Russell fiber bundle, see \aref{def:KRBundle} in \aref{sec:intro}.

For a polynomial $s(x)\in F[x]$ of positive degree and $R(x,y,t)\in F[x,y,t]$, 
we define the closed subscheme $\mathcal{X}(s,R)$ of $\AA^{1}\times\AA^{3}=\Spec(F[x][y,z,t])$ by the equation $\{s(x)z=R(x,y,t)\}$.
The projection map  $\rho=\pr_{x}:\mathcal{X}(s,R) \to \AA^{1}_{x}$ in \eqref{equation:KRfiberbundleprojection} is called a Koras-Russell fiber bundle if
\begin{itemize}
\item[(a)] $\mathcal{X}(s,R)$ is a smooth scheme, and
\item[(b)] For every zero $x_{0}\in\AA^{1}(F)$ of $s(x)$, 
the zero locus in $\AA^2=\Spec(F[y,t])$ of the polynomial $R(x_0,y,t)$ is an integral rational plane curve $\mathcal{C}$ with a unique place at infinity and at most unibranch singularities. 
Equivalently, 
there exists a polynomial $a\in F[y,t]$ and $b(x,y,t)\in F[x,y,t]$ such that 
\[
R(x,y,t)
=
a(y,t)
+
(x-x_{0})b(x,y,t),
\]
such that $\mathcal{C}=\{a(y,t)=0\}$.
\end{itemize}
By definition, $\mathcal{X}(s,R)$ can be realized as the affine modification of $\AA^{1}\times\AA^{2}=\mathrm{Spec}(F[x][y,t])$ along the principal divisor $\ddiv(s)$ 
with center at the closed subscheme defined by the ideal $(s(x),R(x,y,t))$.
The map $\rho:\mathcal{X}(s,R)\to\AA^{1}_{x}$ restricts to the trivial $\AA^{2}$-bundle with total space $\mathrm{Spec}(F[x]_{s(x)}[y,t])$ over the principal open subset 
$\{s(x)\neq0\}$ of $\AA^{1}_{x}$.
For a zero $x_{0}\in\AA^{1}(F)$ of $s(x)$, we have: 
\begin{itemize}
\item[(c)] If the curve $\mathcal{C}:=\{a(y,t)=0\}$ is smooth, then it isomorphic to $\AA^{1}$ and $\rho^{-1}(x_{0})\cong\AA^{2}$. 
\item[(d)] If $\mathcal{C}$ is singular the Lin-Zaidenberg theorem \cite{zbMATH03899059}, \cite{zbMATH04105878}, \cite{zbMATH04105879} implies there exists an automorphism $\varphi^{*}$ of $F[y,t]$ 
such that $\varphi^{*}(a)=y^{\alpha_{1}}+t^{\alpha_{2}}$ for coprime integers $\alpha_{i}\geq 2$. 
Moreover, 
the fiber $\rho^{-1}(x_{0})$ is singular and isomorphic to $\mathcal{C}\times\AA^{1}$. 
\end{itemize}
\begin{example}
\label{exa:DefKR-3fold}
Every deformed Koras-Russell $3$-fold $\mathcal{X}(n,\alpha_{i},p)$ of the first kind \eqref{equation:firstkindpolynomial} is an example of a Koras-Russell fiber bundle.
The unique singular fiber of $\rho:\mathcal{X}(n,\alpha_{i},p)\to\AA^{1}_{x}$ is $\rho^{-1}(0)\cong\Gamma_{\alpha_{1},\alpha_{2}}\times\AA^{1}_{z}$.
\end{example}

The following key geometric result reveals a close connection between Koras-Russell fiber bundles and the existence of Zariski local neighborhoods around rational points on $\AA^{1}$. 
\begin{lemma}
\label{lem:KRBundle-loc-struct}
Suppose $\rho:\mathcal{X}(s,R)\to\AA^{1}_{x}$ is a Koras-Russell fiber bundle and let $x_{0}\in\AA^{1}(F)$. 
\begin{itemize}
\item[(1)] 
If $\rho^{-1}(x_{0})$ is smooth there exists an open neighborhood $\mathcal{U}$ of $x_{0}$ such that $\rho:\mathcal{X}(s,R)|_{\mathcal{U}}\to \mathcal{U}$ is a trivial $\AA^{2}$-bundle. 
\item[(2)] 
If $\rho^{-1}(x_{0})$ is singular there exist an automorphism $\varphi$ of $\AA^{1}$ mapping $x_{0}$ to the origin $0$, 
a deformed Koras-Russell $3$-fold $\mathcal{X}(n,\alpha_{i},p)$ of the first kind \eqref{equation:firstkindpolynomial},
an open neighborhood $\mathcal{U}$ of $x_{0}$,
and an isomorphism $\Psi:\mathcal{X}(s,R)|_{\mathcal{U}}\to \mathcal{X}(n,\alpha_{i},p)|_{\varphi(\mathcal{U})}$ rendering the following diagram commutative: 
\[
\xymatrix{
\mathcal{X}(s,R)|_{\mathcal{U}}  \ar[r]^-{\Psi}\ar[d]_{\rho} & 
\mathcal{X}(n,\alpha_{i},p)|_{\varphi(\mathcal{U})} \ar[d]^-{\pr_{x}} \\ 
\mathcal{U} \ar[r]^-{\varphi} & \varphi(\mathcal{U}). 
}
\]
\end{itemize}
\end{lemma}
\begin{proof}
(1) The assumption implies $\rho^{-1}(x_{0})\cong\AA^{2}$.
Since the generic fiber of $\rho$ is isomorphic to $\AA^{2}$ over the function field $F(x)$ of $\AA^{1}$, 
letting $B=\mathrm{Spec}(\mathcal{O}_{\AA^{1},x_{0}})$ we deduce from \cite[Theorem 1]{zbMATH03855256} that $\mathrm{pr}_{B}:\mathcal{X}(s,R)\times_{\AA^{1}}B\to B$ 
is a trivial $\AA^{2}$-bundle. 
This implies the existence of $\mathcal{U}$ with the claimed properties.

(2) Since $\rho$ restricts to a trivial $\AA^{2}$-bundle over the principal open subset $\mathcal{D}_{s}\subset\AA^{1}$, 
where $s\neq 0$, 
$x_{0}$ is a zero of $s(x)$. 
Up to a coordinate change, we may assume $x_{0}=0$ and write $s(x)=x^{n}\widetilde{s}(x)$, 
where $n\geq1$ and $\widetilde{s}(0)\neq0$. 
This yields $R(x,y,t)=a(y,t)+xb(x,y,t)$ and $\rho^{-1}(0)=\{a(y,t)=0\}\times\mathrm{Spec}(F[z])$. 
Since by hypothesis $\mathcal{C}=\{a(y,t)=0\}$ is a singular curve, 
the Lin-Zaidenberg theorem shows there exists an $F$-automorphism $\varphi^{*}$ of $F[y,t]$ such that $\varphi^{*}(a)=y^{\alpha_{1}}+t^{\alpha_{2}}$
for coprime integers $\alpha_{i}\geq2$. 
This automorphism $\varphi^{*}$ extends an $F[x]$-automorphism $\Phi^{*}$ of $F[x][y,t]$ defined by $\Phi^*(\sum a_{ij}(x)y^it^j)=\sum a_{ij}(x)\varphi^*(y^it^j)$. 
Letting $p(x,y,t)=\Phi^*(b(x,y,t))$, 
we have by construction: 
\[
\Phi^*(R(x,y,t))=\Phi^*(a(y,t)+xb(x,y,t))=y^{\alpha_{1}}+t^{\alpha_{2}}+xp(x,y,t).
\]  
So $\Phi^*$ maps $s(x)$ identically to itself and it maps the ideal $\mathcal{I}=(s(x),R(x,y,t))$ isomorphically to $\mathcal{J}=(s(x),y^{\alpha_{1}}+t^{\alpha_{2}}+xp(x,y,t))$. 
By the universal property of affine modifications \cite[Proposition 2.1]{zbMATH01286199}, 
the corresponding automorphism $\Phi$ of $\AA^3=\Spec(F[x,y,t])$ lifts to an $\AA^{1}_{x}$-isomorphism between $\mathcal{X}(s,R)$ and the Koras-Russell fiber bundle: 
\[
\rho:
\widetilde{\mathcal{X}(s,R)}
:=
\{s(x)z=y^{\alpha_{1}}+t^{\alpha_{2}}+xp(x,y,t)\}
\to\AA^{1}_{x}.
\]
We define the closed subscheme $\mathcal{X}(n,\alpha_{i},p)$ of $\AA^{4}$ by the equation 
\[
\{x^{n}z=y^{\alpha_{1}}+t^{\alpha_{2}}+xp(x,y,t)\},
\]
and let $\mathcal{U}:=\mathcal{D}_{\widetilde{s}}\subset\AA^{1}$ be the principal open subset where $\widetilde{s}\neq 0$.
Then $\widetilde{\mathcal{X}(s,R)}|_{\mathcal{U}}$ and $\mathcal{X}(n,\alpha_{i},p)|_{\mathcal{U}}$ are isomorphic as $\mathcal{U}$-schemes since $\widetilde{s}(x)z$ is a coordinate function on 
$F[x]_{\widetilde{s}(x)}[y,z,t]$ over $F[x]_{\widetilde{s}(x)}$. 
To conclude $\mathcal{X}(n,\alpha_{i},p)$ is a deformed Koras-Russell $3$-fold as in Example \ref{exa:DefKR-3fold}, 
it remains to show $n\geq 2$. 
Since $\mathcal{X}(s,R)|_{\mathcal{U}}$ is smooth, 
$\mathcal{X}(n,\alpha_{i},p)$ is smooth in an open neighborhood of $\mathrm{pr}_x^{-1}(0)$, 
which implies $n\geq 2$ by the Jacobian criterion. 
\end{proof}

\begin{corollary}
\label{corollary:KRBundle-trivial} 
Suppose $\rho:\mathcal{X}(s,R)\to\AA^{1}_{x}$ is a Koras-Russell fiber bundle, 
and let $\mathcal{U}_{0}\subset\AA^{1}$ be the complement of the finitely many points for which the scheme-theoretic fiber of $\rho$ is singular. 
Then $\rho:\mathcal{X}(s,R)|_{\mathcal{U}_{0}}\to \mathcal{U}_{0}$ is a trivial $\AA^{2}$-bundle. 
\end{corollary}
\begin{proof}
By Lemma \ref{lem:KRBundle-loc-struct} (1), 
$\rho:\mathcal{X}(s,R)|_{\mathcal{U}_{0}}\to \mathcal{U}_{0}$ is a Zariski locally trivial $\AA^{2}$-bundle, 
hence isomorphic to a rank $2$ vector bundle by \cite[Theorem 4.4]{zbMATH03577344}, 
which is in fact trivial since $\mathcal{U}_{0}$ is an open subset of $\AA^{1}$.
\end{proof}

\begin{example}
For coprime integers $\alpha_{i}\geq 2$ we define $\mathcal{X}(s,R)\subset\AA^{4}$ by the equation: 
\[
\{x^{2}(x-1)z=(x-1)(y^{\alpha_{1}}+t^{\alpha_{2}})+x(y+t^{2})+x(x-1)\}.
\]
This is a smooth scheme by the Jacobian criterion.
We have $R(0,y,t)=-(y^{\alpha_{1}}+t^{\alpha_{2}})$, 
while for $x=1$:
\[
R(x,y,t)=(y+t^{2})+(x-1)(y^{\alpha_{1}}+t^{\alpha_{2}}+y+t^{2}+x).
\]
Here $\{y^{\alpha_{1}}+t^{\alpha_{2}}=0\}$ and $\{y+t^{2}=0\}$ satisfy hypothesis (1) of \aref{def:KRBundle};
hence $\rho:\mathcal{X}(s,R)\to\AA^{1}_{x}$ is a Koras-Russell fiber bundle. 
Furthermore, 
$\rho$ restricts to a trivial $\AA^{2}$-bundle over $\AA^{1}\setminus\{0\}$ by \aref{corollary:KRBundle-trivial}.
\end{example}

\begin{lemma}
\label{lem:KRBundle-modification}
Suppose $\rho:\mathcal{X}(s,R)\to\AA^{1}_{x}$ is a Koras-Russell fiber bundle, 
and $\rho^{-1}(x_{0})\cong \mathcal{C}\times\AA^{1}$ is singular for the rational point $x_{0}\in\AA^{1}(F)$.
Let $\pi:\widetilde{\mathcal{X}(s,R)}\to \mathcal{X}(s,R)$ be the affine modification of $\mathcal{X}(s,R)$ along the divisor $\{x=x_{0}\}$ with center $\{c\}\times\AA^{1}\subset\rho^{-1}(x_{0})$ 
over a smooth point $c$ on the curve $\mathcal{C}$.
\begin{itemize}
\item[(1)]
Then $\widetilde{\rho}=\rho\circ\pi:\widetilde{\mathcal{X}(s,R)}\to\AA^{1}_{x}$ is a Koras-Russell fiber bundle with $\widetilde{\rho}^{-1}(x_{0})\cong\AA^{2}$.
\item[(2)]
By restricting $\pi$ to $\widetilde{\mathcal{X}(s,R)}\setminus\widetilde{\rho}^{-1}(x_{0})$ we obtain an isomorphism and a commutative diagram: 
\[
\xymatrix{
\widetilde{\mathcal{X}(s,R)} \setminus\widetilde{\rho}^{-1}(x_{0}) \ar[rr]^-{\cong}\ar[dr]_-{\widetilde{\rho}} && \mathcal{X}(s,R)\setminus{\rho}^{-1}(x_{0}) \ar[dl]^-{\rho} \\ 
& \AA^{1}. 
}
\]
\end{itemize}
\end{lemma}
\begin{proof}
(1) Up to a change of coordinates, 
we may assume $x_{0}=0$, 
$s(x)=x^{n}\widetilde{s}(x)$, $n\geq 2$, $\widetilde{s}(0)\neq0$, 
and:
\[
R(x,y,t)
=
y^{\alpha_{1}}-t^{\alpha_{2}}+xb(x,y,t).
\]
Furthermore,
we may assume the center is the line over the point $c=(1,1)$ on the curve $\mathcal{C}=\{y^{\alpha_{1}}=t^{\alpha_{2}}\}$.
Since $\mathcal{X}(s,R)$ and the center $\{c\}\times\AA^{1}$ of the affine modification $\pi:\widetilde{\mathcal{X}(s,R)}\to \mathcal{X}(s,R)$ are smooth,
so is $\widetilde{\mathcal{X}(s,R)}$. 
It remains to check $\widetilde{\rho}^{-1}(0)\cong\AA^{2}$.
Setting $y=xy_{1}+1$ and $t=xt_{1}+1$ yields:
\begin{align*}
y^{\alpha_{1}}-t^{\alpha_{2}}+xb(x,y,t) & =x\left(\alpha_1y_{1}-\alpha_2t_{1})+b(0,1,1)\right)+x^{2}\widetilde{b}(x,y_{1},t_{1}).
\end{align*}
Thus the equation defining $\widetilde{\mathcal{X}(s,R)}$ in $\AA^{4}$ takes the form:
\[
\{x^{n-1}\widetilde{s}(x)z=x^{-1}R(x,xy_{1}+1,xt_{1}+1)=\alpha_1y_{1}-\alpha_2t_{1}+b(0,1,1)+x\widetilde{b}(x,y_{1},t_{1})\}.
\]
This implies $\widetilde{\rho}^{-1}(0)\cong\left\{\alpha_1y_{1}=\alpha_2t_{1}+b(0,1,1)\right\} \times\AA^{1}$ is the fiber product of an affine line in $\AA^{2}$ by $\AA^{1}$, 
i.e., 
it is isomorphic to $\AA^{2}$. 
Part (2) follows readily. 
\end{proof}

\subsection{$\AA^{1}$-contractibility of deformed Koras-Russell $3$-folds after simplicial suspension}
In this section we refine the results and techniques in \cite[\S3.2]{MRDF}.
\begin{lemma}
\label{lem:DefKR-Punctured-Weak-Equiv}
If $\mathcal{X}(n,\alpha_{i},p)$ is a deformed Koras-Russell $3$-fold of the first kind \eqref{equation:firstkindpolynomial}, 
then: 
$$
\mathcal{X}(n,\alpha_{i},p)
\setminus 
(\{x=y=t=0\}\cong\AA^{1}_{z})
\sim_{\AA^{1}}
\AA^{2}\setminus\{(0,0)\}.
$$
\end{lemma}
\begin{proof}
Using a coordinate change we may rewrite the defining equation for $\mathcal{X}(n,\alpha_{i},p)$ as:
\[
\mathcal{X}(\alpha_{2})
:=
\{x^{n}z=y^{\alpha_{1}}-t^{\alpha_{2}}+xp(x,y,t)\}, 
\alpha_{2}\geq 1.
\]
Let $\pi:=\mathrm{pr}_{x,t}:\mathcal{X}(\alpha_{2})\to\AA^{2}$ and define $\varphi_{\alpha_{2}}:\AA^{2}=\mathrm{Spec}(F[u,v])\to\AA^{2}$ by 
$(u,v)\mapsto(x,t)=(u^{\alpha_{1}}-v^{\alpha_{2}},v)$. 
Here $\varphi_{\alpha_{2}}$ is a cyclic Galois cover, 
fully ramified over the line $\{t=0\}\subset\AA^{2}$ and \'etale elsewhere. 
The fiber product $\mathcal{X}(\alpha_{2})\times_{\AA^{2}_{x,t}}\AA^{2}_{u,v}$ is isomorphic to the closed subscheme of $\AA^{4}=\mathrm{Spec}(F[u,v][y,z])$ defined by: 
\[
\widetilde{\mathcal{X}}(\alpha_{2})
:=
\{(u^{\alpha_{1}}-v^{\alpha_{2}})^{n}z=y^{\alpha_{1}}-v^{\alpha_{2}}+(u^{\alpha_{1}}-v^{\alpha_{2}})p(u^{\alpha_{1}}-v^{\alpha_{2}},y,v)\}.
\]
The Galois group $\mu_{\alpha_{1}}$ of the cover $\varphi_{\alpha_{2}}$ of $\widetilde{\mathcal{X}}(\alpha_{2})$ is generated by $(u,v,y,z)\mapsto(\varepsilon u,v,y,z)$,
where $\varepsilon\in F^{*}$ is a primitive $\alpha_{1}$-th root of unity. 
The projection $\widetilde{\pi}:=\mathrm{pr}_{u,v}:\widetilde{\mathcal{X}}(\alpha_{2})\to\AA^{2}$ restricts to a trivial $\AA^{1}$-bundle over the complement of the irreducible curve 
$\Gamma_{\alpha_{1},\alpha_{2}}:=\{u^{\alpha_{1}}-v^{\alpha_{2}}=0\}$. 
The fiber of $\widetilde{\pi}$ over $(0,0)$ is supported by the line $\widetilde{\AA^{1}_{z}}:=\{u=v=y=0\}$, 
while $\widetilde{\pi}^{-1}(\Gamma_{\alpha_{1},\alpha_{2}})$ is isomorphic to: 
\[
\mathrm{Spec}(F[u,v][y,z]/(u^{\alpha_{1}}-v^{\alpha_{2}},y^{\alpha_{1}}-u^{\alpha_{1}})).
\]
The normalization map $\nu:\widetilde{\Gamma_{\alpha_{1},\alpha_{2}}}=\mathrm{Spec}(F[w])\to \Gamma_{\alpha_{1},\alpha_{2}}$,
$w\mapsto(w^{\alpha_{2}},w^{\alpha_{1}})$, 
restricts to an isomorphism between $\widetilde{\Gamma_{\alpha_{1},\alpha_{2}}}\setminus\{(0,0)\}=\AA_{*}^{1}$ and $\Gamma_{\alpha_{1},\alpha_{2}}\setminus\{(0,0)\}$.
Moreover,
$\widetilde{\pi}^{-1}(\Gamma_{\alpha_{1},\alpha_{2}}\setminus\{(0,0)\})$ splits into a disjoint union of $\alpha_{1}$ irreducible components 
$\mathcal{D}_{\varepsilon^{i}}:=\{u^{\alpha_{1}}-v^{\alpha_{2}}=y-\varepsilon^{i}u\}\setminus\{u=v=y=0\}$,
$0\leq i\leq\alpha_{1}-1$, 
all isomorphic to $\AA_{*}^{1}\times\mathrm{Spec}(F[z])$.
It follows that $\widetilde{\pi}$ restricts to a smooth affine map
\[
\widetilde{\pi}_{i}:
\widetilde{\mathcal{X}}_{i}:=\widetilde{\mathcal{X}}(\alpha_{2})\setminus\widetilde{\AA^{1}_{z}}\cup\bigcup_{j\neq i}\mathcal{D}_{\varepsilon^{j}}
\to
\AA^{2}\setminus\{(0,0)\},
\]
and a locally trivial $\AA^{1}$-bundle since all its fibers are isomorphic to $\AA^{1}$ \cite[Theorem 1]{MR503968}.  

We let $\widetilde{\delta}:\widetilde{\mathcal{S}}\to\AA^{2}\setminus\{(0,0)\}$ denote the smooth scheme obtained by gluing the $\alpha_{1}$ copies 
\[
\widetilde{\mathcal{S}}_{\varepsilon^{i}}:=\mathrm{Spec}(F[u_{\varepsilon^{i}},v])\setminus\{(0,0)\}, 
0\leq i\leq\alpha_{1}-1,
\] 
of $\AA^{2}\setminus\{(0,0)\}$ by the identity map away from the cuspidal curves $\{u_{\varepsilon^{i}}^{\alpha_{1}}=v^{\alpha_{2}}\}$.
With this definition, 
\[
\widetilde{\pi}:
\widetilde{\mathcal{X}}(\alpha_{2})\setminus\widetilde{\AA^{1}_{z}}
\to
\AA^{2}\setminus\{(0,0)\}
\]
factors through the locally trivial $\AA^{1}$-bundle: 
\[
\widetilde{\rho}_{\alpha_{2}}:
\widetilde{\mathcal{X}}(\alpha_{2})\setminus\widetilde{\AA^{1}_{z}}
\to
\widetilde{\mathcal{S}}.
\]
The action of the Galois group $\mu_{\alpha_{1}}$ lifts readily to an action on $\widetilde{\mathcal{S}}$ defined by sending $(u_{\varepsilon^{i}},v)\in \widetilde{\mathcal{S}}_{\varepsilon^{i}}$ to 
$(\varepsilon u_{\varepsilon^{i+1}},v)\in\widetilde{\mathcal{S}}_{\varepsilon^{i+1}}$.
This yields $\mu_{\alpha_{1}}$-equivariant maps $\widetilde{\rho}_{\alpha_{2}}:\widetilde{\mathcal{X}}(\alpha_{2})\setminus\widetilde{\AA^{1}_{z}}\to\widetilde{\mathcal{S}}$ and 
$\widetilde{\delta}:\widetilde{\mathcal{S}}\to\AA^{2}\setminus\{(0,0)\}$. 

As algebraic spaces, 
the geometric quotients of $\widetilde{\mathcal{X}}(\alpha_{2})\setminus\widetilde{\AA^{1}_{z}}$ and $\AA^{2}\setminus\{(0,0)\}$ under the $\mu_{\alpha_{1}}$-actions are isomorphic to 
$\mathcal{X}(\alpha_{2})\setminus \AA^{1}_{z}$ respectively $\AA^{2}\setminus(\{0,0\})=\mathrm{Spec}(F[x,t])\setminus\{(0,0)\}$.
The geometric quotient for the $\mu_{\alpha_{1}}$-action on $\widetilde{\mathcal{S}}$ is an algebraic space $\mathfrak{\mathcal{S}}(\alpha_{1},\alpha_{2})$ (see e.g., \cite{MR3084720}). 
The map $\delta:\mathfrak{\mathcal{S}}(\alpha_{1},\alpha_{2})\to\AA^{2}\setminus(\{0,0\})$ induced by $\widetilde{\delta}$ has the following properties: 
\begin{itemize}
\item
$\delta$ restricts to an isomorphism away from the punctured affine line $\{x=0\}\cong\mathrm{Spec}(F[t^{\pm1}])$.
\item
The restriction of $\delta$ to $\delta^{-1}(\{x=0\})$ is an \'etale cyclic Galois cover of order $\alpha_{1}$ isomorphic to the projection 
$\mathrm{pr}_{t}:\mathrm{Spec}(F[t^{\pm1}][u]/(u^{\alpha_{1}}-t^{\alpha_{2}}))\to\mathrm{Spec}(F[t^{\pm1}])$.
\end{itemize}
Likewise, 
$\widetilde{\rho}_{\alpha_{2}}:\widetilde{\mathcal{X}}(\alpha_{2})\setminus\widetilde{\AA^{1}_{z}}\to\widetilde{\mathcal{S}}$ descends to an \'etale locally trivial $\AA^{1}$-bundle 
$\rho_{\alpha_{2}}:\mathcal{X}(\alpha_{2})\setminus \AA^{1}_{z}\to\mathfrak{\mathcal{S}}(\alpha_{1},\alpha_{2})$.
While the ramified Galois cover $\varphi_{\alpha_{2}}:\AA^{2}\to\AA^{2}$ depends on $\alpha_{1}$ and $\alpha_{2}$, 
the algebraic space $\delta:\mathfrak{\mathcal{S}}(\alpha_{1},\alpha_{2})\to\AA^{2}\setminus(\{0,0\})$ depends only on $\alpha_{1}$ as explained in \cite[Lemma 3.2]{MRDF} 
and \cite[\S1.1]{MR3285620}.
When $\alpha_{2}\geq2$, 
$\mathcal{X}(\alpha_{2})\setminus \AA^{1}_{z}$ and $\mathcal{X}(1)\setminus \AA^{1}_{z}$ are \'etale locally trivial $\AA^{1}$-bundles over the algebraic space 
$\mathfrak{\mathcal{S}}(\alpha_{1},\alpha_{2})=\mathfrak{\mathcal{S}}(\alpha_{1},1)=\mathfrak{\mathcal{S}}$;
in general, this is not a scheme (see \cite[Remark 1.1]{MR3285620}).
As a consequence, 
the fiber product 
\[
W:=
(\mathcal{X}(\alpha_{2})\setminus \AA^{1}_{z})\times_{\mathfrak{\mathcal{S}}}(\mathcal{X}(1)\setminus \AA^{1}_{z})
\]
is simultaneously an \'etale locally trivial $\AA^{1}$-bundle over $\mathcal{X}(\alpha_{2})\setminus \AA^{1}_{z}$ and $\mathcal{X}(1)\setminus \AA^{1}_{z}$ via the first and second projection, 
respectively. 
Since $\mathcal{X}(\alpha_{2})\setminus \AA^{1}_{z}$ is a scheme for $\alpha_{2}\geq1$, 
these bundles are locally trivial in the Zariski topology. 
It follows that $\mathrm{pr}_{1}:W\to \mathcal{X}(\alpha_{2})\setminus \AA^{1}_{z}$ and $\mathrm{pr}_{2}:W\to \mathcal{X}(1)\setminus \AA^{1}_{z}$ are $\AA^{1}$-weak equivalences. 

To complete the proof it remains to show $\mathcal{X}(1)\setminus \AA^{1}_{z}$ is $\AA^{1}$-weakly equivalent to $\AA^{2}\setminus\{(0,0)\}$.
When $\alpha_{2}=1$, 
the fiber $\mathrm{pr}_{x}^{-1}(0)$ of $\mathrm{pr}_{x}:\mathcal{X}(1)\to\AA^{1}_{x}$ is isomorphic to $\AA^{2}=\mathrm{Spec}(F[y,z])$,
and $\mathrm{pr}_{x}:\mathcal{X}(1)\to\AA^{1}_{x}$ is a trivial $\AA^{2}$-bundle by virtue of \cite{zbMATH03855256}. 
Furthermore, 
since $\AA^{1}_{z}\subset\mathrm{pr}_{x}^{-1}(0)$ is the line $\{y=0\}$, 
the chain of inclusions $\AA^{1}_{z}\subset\mathrm{pr}_{x}^{-1}(0)\subset \mathcal{X}(1)$ identifies with the linear subspaces $\{x=\alpha=0\}\subset\{x=0\}\subset\AA^{1}\times\AA^{2}$.
The projection $\mathrm{pr}_{x,\alpha}:\AA^{1}\times\AA^{2}\setminus\{x=\alpha=0\}\to\AA^{2}\setminus\{(0,0)\}$
yields the desired $\AA^{1}$-weak equivalence $\mathcal{X}(1)\setminus \AA^{1}_{z}\sim_{\AA^{1}}\AA^{2}\setminus\{(0,0)\}$. 
\end{proof}

The proof of the next result combines the techniques used in \cite{MRDF} and \cite{MR3549169}. 
\begin{theorem}
\label{theorem:DefKR-StableA1-cont} 
If $\mathcal{X}(n,\alpha_{i},p)$ is a deformed Koras-Russell $3$-fold of the first kind \eqref{equation:firstkindpolynomial}, 
then: 
$$
\Sigma_{s} \mathcal{X}(n,\alpha_{i},p)
\sim_{\AA^{1}}
\ast.
$$ 
 \end{theorem}
\begin{proof}
Recall $\mathcal{X}(n,\alpha_{i},p):=\{x^{n}z=y^{\alpha_{1}}+t^{\alpha_{2}}+xp(x,y,t)\}\subset\AA^{4}$ and $\AA^{1}_{z}\cong \{x=y=t=0\}\subset \mathcal{X}(n,\alpha_{i},p)$. 
Lemma \ref{lem:DefKR-Punctured-Weak-Equiv} shows $\mathcal{X}(n,\alpha_{i},p)\setminus \AA^{1}_{z}$ is $\AA^{1}$-weakly equivalent to $\AA^{2}\setminus\{(0,0)\}$. 
Since $\AA^{1}_{z}$ has trivial normal bundle in $\mathcal{X}(n,\alpha_{i},p)$, 
homotopy purity \cite[Theorem 3.2.23]{MR1813224} implies $\AA^{1}$-weak equivalences:
\[
\mathcal{X}(n,\alpha_{i},p)/(\mathcal{X}(n,\alpha_{i},p)\setminus \AA^{1}_{z})
\sim_{\AA^{1}}
(\AA^{1}_{z})_{+}\wedge(\mathbb{P}^{1})^{\wedge2}
\sim_{\AA^{1}}
(\mathbb{P}^{1})^{\wedge2}.
\]
Since $\Sigma_{s}\AA^{2}\setminus\{(0,0)\}$ is $\AA^{1}$-weakly equivalent to $S^{2}\wedge(\mathbb{G}_{m})^{\wedge2}\sim_{\AA^{1}} (\mathbb{P}^{1})^{\wedge2}$, 
the cofiber sequence 
\begin{equation}
\label{equation:cofibersequencedeformed1kind}
\mathcal{X}(n,\alpha_{i},p)\setminus \AA^{1}_{z}\to \mathcal{X}(n,\alpha_{i},p)\to \mathcal{X}(n,\alpha_{i},p)/(\mathcal{X}(n,\alpha_{i},p)\setminus \AA^{1}_{z})
\end{equation}
induces via simplicial suspension the connecting map:  
\begin{equation}
\label{equation:connectingfirstkind}
\mathcal{X}(n,\alpha_{i},p)/(\mathcal{X}(n,\alpha_{i},p)\setminus \AA^{1}_{z})
\sim_{\AA^{1}}
(\mathbb{P}^{1})^{\wedge2}
\stackrel{\partial}{\to}
\Sigma_{s} \mathcal{X}(n,\alpha_{i},p)\setminus \AA^{1}_{z}
\sim_{\AA^{1}}
\Sigma_{s}\AA^{2}\setminus\{(0,0)\}
\sim_{\AA^{1}}
(\mathbb{P}^{1})^{\wedge2}.
\end{equation}
To conclude that $\Sigma_{s} \mathcal{X}(n,\alpha_{i},p)$ is $\AA^{1}$-contractible, it suffices to show \eqref{equation:connectingfirstkind} is an $\AA^{1}$-weak equivalence.
There are isomorphisms for Milnor-Witt $K$-theory sheaves
\[
H^{1}(\mathcal{X}(n,\alpha_{i},p)\setminus \AA^{1}_{z},\mathbf{K}_{2}^{\mathrm{MW}})
\cong 
H^{1}(\AA^{2}\setminus\{(0,0)\},\mathbf{K}_{2}^{\mathrm{MW}})
\cong
\mathbf{K}_{0}^{\mathrm{MW}}(F)
\cong 
H^{2}((\mathbb{P}^{1})^{\wedge2},\mathbf{K}_{2}^{\mathrm{MW}}),
\]
so that the long exact sequence in cohomology associated to \eqref{equation:cofibersequencedeformed1kind} takes the form:
\[
\cdots
\to 
H^{1}(\mathcal{X}(n,\alpha_{i},p),\mathbf{K}_{2}^{\mathrm{MW}})
\to
\mathbf{K}_{0}^{\mathrm{MW}}(F)
\stackrel{\alpha}{\to}\mathbf{K}_{0}^{\mathrm{MW}}(F)
\to 
H^{2}(\mathcal{X}(n,\alpha_{i},p),\mathbf{K}_{2}^{\mathrm{MW}})
\to
\cdots.
\]
The $\mathbf{K}_{0}^{\mathrm{MW}}(F)$-linear map $\alpha$ is multiplication by the element of $\mathbf{K}_{0}^{\mathrm{MW}}(F)$ corresponding to $\partial$ in \eqref{equation:connectingfirstkind} 
under the isomorphism $[(\mathbb{P}^{1})^{\wedge2},(\mathbb{P}^{1})^{\wedge2}]_{\AA^{1}}\cong\mathbf{K}_{0}^{\mathrm{MW}}(F)$.
We claim $\alpha$ is an isomorphism, 
in fact: 
\[
H^{1}(\mathcal{X}(n,\alpha_{i},p),\mathbf{K}_{2}^{\mathrm{MW}})=H^{2}(\mathcal{X}(n,\alpha_{i},p),\mathbf{K}_{2}^{\mathrm{MW}})=0.
\]
Since $H^{i+m}((\mathbb{P}^{1})^{\wedge m},\mathbf{K}_{2+m}^{\mathrm{MW}})=0$ for $i\geq1$, the fact that there is a cofiber sequence 
\[(\PP^1)^{\wedge m}\to \mathcal{X}(n,\alpha_{i},p)_{+}\wedge (\PP^1)^{\wedge m}\to \mathcal{X}(n,\alpha_{i},p)\wedge(\mathbb{P}^{1})^{\wedge m}\] 
and
\[
H^{i}(\mathcal{X}(n,\alpha_{i},p),\mathbf{K}_{2}^{\mathrm{MW}})
\cong 
H^{i+m}(\mathcal{X}(n,\alpha_{i},p)_{+}\wedge(\mathbb{P}^{1})^{\wedge m},\mathbf{K}_{2+m}^{\mathrm{MW}})
\]
for $i,m\geq1$ \cite{zbMATH06170873}, 
it suffices to show $\mathcal{X}(n,\alpha_{i},p)\wedge(\mathbb{P}^{1})^{\wedge m}\sim_{\AA^{1}}\ast$ for $m\gg 0$. 
This follows from \aref{theorem:higherChowGroups} since for every smooth affine scheme $\mathcal{Y}$ the pullback map
\[
\mathrm{CH}_{\ast}(\mathcal{Y})
\to
\mathrm{CH_{\ast}}(\mathcal{X}(n,\alpha_{i},p)\times\mathcal{Y})
\]
is an isomorphism, 
owing to \aref{corollary:deformed KR stable}.
\end{proof}

Let $\sigma=\mathrm{pr}_{x,y,t}:\mathcal{X}(n,\alpha_{i},p)\rightarrow \AA^3$ be the morphism expressing $\mathcal{X}(n,\alpha_{i},p)$ as the affine modification of $\AA^3$ along $\ddiv (x^n)$ along 
$\ddiv(x^n)$ with center at the closed subscheme defined by the ideal $(x^{n},y^{\alpha_{1}}+t^{\alpha_{2}}+xp(x,y,t))$. 
Let $\pi:\widetilde{\mathcal{X}(n,\alpha_{i},p)}\to \mathcal{X}(n,\alpha_{i},p)$ be the affine modification of $\mathcal{X}(n,\alpha_{i},p)$ along $\ddiv(x)$  with center 
$\widetilde{Z}:\{q\}\times\AA^{1}\subset\rho^{-1}(0)$ for a smooth point $q$ on the cuspidal curve $\Gamma_{\alpha_{1},\alpha_{2}}$. 
\aref{corollary:KRBundle-trivial} and \aref{lem:KRBundle-modification} imply the composite 
\[
\widetilde{\rho}=\rho\circ\pi:\widetilde{\mathcal{X}(n,\alpha_{i},p)}\to\AA^{1}_{x}
\]
is a trivial $\AA^{2}$-bundle. 
This yields a commutative diagram of horizontal cofiber sequences: 
\[
\xymatrix{
\widetilde{\mathcal{X}(n,\alpha_{i},p)}\setminus\widetilde{\rho}^{-1}(0) \ar[r]\ar[d]_{\cong} & 
\widetilde{\mathcal{X}(n,\alpha_{i},p)} \ar[r]\ar[d]_{\pi} &
\widetilde{\mathcal{X}(n,\alpha_{i},p)}/(\widetilde{\mathcal{X}(n,\alpha_{i},p)}\setminus\widetilde{\rho}^{-1}(0))\sim_{\AA^{1}}\AA_{+}^{2}\wedge\mathbb{P}^{1}\sim_{\AA^{1}}\mathbb{P}^{1} \ar[d]_{\overline{\pi}}\\ 
\mathcal{X}(n,\alpha_{i},p)\setminus\rho^{-1}(0) \ar[r]\ar[d]_{\cong} & 
\mathcal{X}(n,\alpha_{i},p)  \ar[r]\ar[d]_{\sigma}& \mathcal{X}(n,\alpha_{i},p)/(\mathcal{X}(n,\alpha_{i},p)\setminus\rho^{-1}(0))\ar[d]_{\overline{\sigma}}\\
\AA^{1}\setminus\{0\}\times\AA^{2} \ar[r]& 
\AA^{1}\times\AA^{2} \ar[r]&
\AA^{1}\times\AA^{2}/(\AA^{1}\setminus\{0\}\times\AA^{2})\sim_{\AA^{1}}\AA_{+}^{2}\wedge\mathbb{P}^{1}\sim_{\AA^{1}}\mathbb{P}^{1}.
}
\]
Since $\widetilde{\mathcal{X}(n,\alpha_{i},p)}\cong\AA^{1}\times\AA^{2}$ is $\AA^{1}$-contractible and $\Sigma_{s} \mathcal{X}(n,\alpha_{i},p)$ 
is $\AA^{1}$-contractible by \aref{theorem:DefKR-StableA1-cont}, 
we obtain the induced $\AA^{1}$-weak equivalences
\begin{equation}
\label{equation:sigma1firstkind}
\Sigma_{s}\pi:\Sigma_{s}\widetilde{\mathcal{X}(n,\alpha_{i},p)}
\overset{\sim_{\AA^{1}}}{\to}
\Sigma_{s} \mathcal{X}(n,\alpha_{i},p),
\end{equation}
and: 
\begin{equation}
\label{equation:sigma2firstkind}
\Sigma_{s}\sigma:\Sigma_{s} \mathcal{X}(n,\alpha_{i},p)
\overset{\sim_{\AA^{1}}}{\to}
\Sigma_{s}(\AA^{1}\times\AA^{2}).
\end{equation}
\begin{corollary}
\label{corollary:ThomSpacesIso} 
There are naturally induced $\AA^{1}$-weak equivalences 
\[
\Sigma_{s}\overline{\pi}:\Sigma_{s}\mathbb{P}^{1}
\overset{\sim_{\AA^{1}}}{\to}
\Sigma_{s}(\mathcal{X}(n,\alpha_{i},p)/(\mathcal{X}(n,\alpha_{i},p)\setminus\rho^{-1}(0)),
\]
and: 
\[
\Sigma_{s}\overline{\sigma}:
\Sigma_{s}(\mathcal{X}(n,\alpha_{i},p)/(\mathcal{X}(n,\alpha_{i},p)\setminus\rho^{-1}(0))
\overset{\sim_{\AA^{1}}}{\to}
\Sigma_{s}\mathbb{P}^{1}.
\]
\end{corollary}

\subsection{$\AA^{1}$-contractibility of Koras-Russell fiber bundles after simplicial suspension}
In this section we prove \aref{thm:KRBundle-Stable-A1Cont} stated in the introduction.
\begin{theorem} 
\label{thm:KRBundle-Stable-A1Contsimplicual}
Let $\rho:\mathcal{X}(s,R)\to\AA^{1}_{x}$ be a Koras-Russell fiber bundle. 
Then $\Sigma_{s} \mathcal{X}(s,R)$ is $\AA^{1}$-contractible. 
\end{theorem}
\begin{proof}
We proceed by induction on the number of singular fibers of $\rho:\mathcal{X}(s,R)\to\AA^{1}_{x}$.
If there are no singular fibers, 
then $\rho$ is a trivial $\AA^{2}$-bundle by \aref{corollary:KRBundle-trivial}, 
so that $\mathcal{X}(s,R)$ is trivially $\AA^{1}$-contractible. 
Suppose $x_{0}\in\AA^{1}(F)$ is a point with singular fiber $\rho^{-1}(x_{0})$. 
As in \aref{lem:KRBundle-modification} the affine modification 
\[
\pi:\widetilde{\mathcal{X}(s,R)}\to \mathcal{X}(s,R)
\]
yields a Koras-Russell fiber bundle with one less singular fiber than $\rho:\mathcal{X}(s,R)\to\AA^{1}_{x}$,
namely:
\[
\widetilde{\rho}=\rho\circ\pi:\widetilde{\mathcal{X}(s,R)}\to\AA^{1}_{x}.
\]
It suffices to show that if $\Sigma_{s}\widetilde{\mathcal{X}(s,R)}$ is $\AA^{1}$-contractible,
then so is $\Sigma_{s} \mathcal{X}(s,R)$. 
By \aref{lem:KRBundle-loc-struct}(2),
there exists an open neighborhood $\mathcal{U}$ of $0$, 
a deformed Koras-Russell $3$-fold of the first kind $\pr_{x}:\mathcal{X}(n,\alpha_{i},p)\to\AA^{1}_{x}$, 
and an isomorphism of $\mathcal{U}$-schemes $\Psi:\mathcal{X}(s,R)|_{\mathcal{U}}\stackrel{\cong}{\to}\mathcal{X}(n,\alpha_{i},p)|_{\mathcal{U}}$. 
We form the affine modification 
\[
\pi':\widetilde{\mathcal{X}(n,\alpha_{i},p)}\to \mathcal{X}(n,\alpha_{i},p)
\]
along the principal divisor $\pr_{x}^{-1}(x_{0})$ with center in the image of $\Psi$ restricted to the center of $\pi$, 
and set $\widetilde{\pr_{x}}=\pr_{x}\circ\pi':\widetilde{\mathcal{X}(n,\alpha_{i},p)}\to\AA^{1}_{x}$.
By \cite[Proposition 2.1]{zbMATH01286199} $\Psi$ lifts to an isomorphism of $\mathcal{U}$-schemes:
\[
\widetilde{\Psi}:\widetilde{\mathcal{X}(s,R)}|_{\mathcal{U}}
\stackrel{\cong}{\to}
\widetilde{\mathcal{X}(n,\alpha_{i},p)}|_{\mathcal{U}}.
\]
Invoking excision, 
$\Psi$ induces an $\AA^{1}$-weak equivalence 
\[
\theta:
\mathrm{Th}:=\mathcal{X}(s,R)/(\mathcal{X}(s,R)\setminus\rho^{-1}(x_{0}))
\stackrel{\sim_{\AA^{1}}}{\to}
\mathrm{Th}':=\mathcal{X}(n,\alpha_{i},p)/(\mathcal{X}(n,\alpha_{i},p)\setminus \pr_{x}^{-1}(x_{0}))
\] 
between the homotopy cofibers of the open inclusions $\mathcal{X}(s,R)\setminus\rho^{-1}(x_{0})\hookrightarrow \mathcal{X}(s,R)$ and 
$\mathcal{X}(n,\alpha_{i},p)\setminus \pr_{x}^{-1}(x_{0})\hookrightarrow \mathcal{X}(n,\alpha_{i},p)$, 
respectively. 
Similarly,
$\widetilde{\Psi}$ induces an $\AA^{1}$-weak equivalence 
\[
\widetilde{\theta}:
\widetilde{\mathrm{Th}}:=\widetilde{\mathcal{X}(s,R)}/(\widetilde{\mathcal{X}(s,R)}\setminus\widetilde{\rho}^{-1}(x_{0}))
\stackrel{\sim_{\AA^{1}}}{\to}
\widetilde{\mathrm{Th}}':=\widetilde{\mathcal{X}(n,\alpha_{i},p)}/(\widetilde{\mathcal{X}(n,\alpha_{i},p)}\setminus\widetilde{\pr_{x}}^{-1}(x_{0}))
\] 
between the homotopy cofibers of the open inclusions $\widetilde{\mathcal{X}(s,R)}\setminus\widetilde{\rho}^{-1}(x_{0})\hookrightarrow\widetilde{\mathcal{X}(s,R)}$ and 
$\widetilde{\mathcal{X}(n,\alpha_{i},p)}\setminus\widetilde{\pr_{x}}^{-1}(x_{0})\hookrightarrow\widetilde{\mathcal{X}(n,\alpha_{i},p)}$, 
respectively. 
Letting $\overline{\pi}:\widetilde{\mathrm{Th}}\to\mathrm{Th}$ and $\overline{\pi}':\widetilde{\mathrm{Th}}'\to\mathrm{Th}'$ denote the natural maps induced by $\pi$ and $\pi'$, 
respectively,
we obtain the commutative diagram: 
\[
\xymatrix{
\widetilde{\mathrm{Th}}  \ar[r]^-{\widetilde{\theta}}\ar[d]_-{\overline{\pi}} & 
\widetilde{\mathrm{Th}}'\ar[d]^-{\overline{\pi}'}\\ 
\mathrm{Th}  \ar[r]^-{\theta} & \mathrm{Th'}.
}
\]
The map $\pi'$ gives rise to a commutative diagram: 
\[
\xymatrix{
\cdots \ar[r]&\Sigma_{s}\widetilde{\mathcal{X}(n,\alpha_{i},p)}\setminus\widetilde{\pr_{x}}^{-1}(x_{0})  \ar[r]\ar[d]_{\cong} & 
\Sigma_{s}\widetilde{\mathcal{X}(n,\alpha_{i},p)} \ar[d]_{\Sigma_{s}\pi'}\ar[r]&\Sigma_{s}\widetilde{\mathrm{Th}}' \ar[d]_{\Sigma_{s}\overline{\pi}'}\ar[r]&\cdots \\ 
\cdots \ar[r]&\Sigma_{s} \mathcal{X}(n,\alpha_{i},p)\setminus \pr_{x}^{-1}(x_{0}) \ar[r] & \Sigma_{s} \mathcal{X}(n,\alpha_{i},p)\ar[r] & \Sigma_{s}\mathrm{Th}'\ar[r]& \cdots.
}
\]
\aref{corollary:KRBundle-trivial} shows $\widetilde{\pr_{x}}:\widetilde{\mathcal{X}(n,\alpha_{i},p)}\to\AA^{1}_{x}$ is a trivial $\AA^{2}$-bundle, 
so that $\Sigma_{s}\widetilde{\mathcal{X}(n,\alpha_{i},p)}$ is $\AA^{1}$-contractible. 
Here $\Sigma_{s}\overline{\pi}':\Sigma_{s}\widetilde{\mathrm{Th}}'\to\Sigma_{s}\mathrm{Th}'$ and hence 
$\Sigma_{s}\overline{\pi}:\Sigma_{s}\widetilde{\mathrm{Th}}\to\Sigma_{s}\mathrm{Th}$ are $\AA^{1}$-weak equivalence by \aref{corollary:ThomSpacesIso}.
In the commutative diagram of cofiber sequences
\[
\xymatrix{
\Sigma_{s}\widetilde{\mathcal{X}(s,R)}\setminus\widetilde{\rho}^{-1}(x_{0}) \ar[r]\ar[d]_{\cong} & 
\Sigma_{s}\widetilde{\mathcal{X}(s,R)}  \ar[d]_{\Sigma_{s}\pi}\ar[r]&\Sigma_{s}\widetilde{\mathrm{Th}} \ar[d]_{\Sigma_{s}\overline{\pi}} \\ 
\Sigma_{s} \mathcal{X}(s,R)\setminus\rho^{-1}(x_{0}) \ar[r] & \Sigma_{s} \mathcal{X}(s,R)\ar[r] & \Sigma_{s}\mathrm{Th},
}
\]
$\Sigma_{s}\widetilde{\mathcal{X}(s,R)}$ is $\AA^{1}$-contractible by the induction hypothesis. 
The very weak five lemma in pointed model categories \cite[Lemma 2.1]{MRDF} implies $\Sigma_{s}\pi$ is an $\AA^{1}$-weak equivalence, 
and hence $\Sigma_{s}\mathcal{X}(s,R)\sim_{\AA^{1}}\ast$. 
\end{proof}

\subsection{$\AA^{1}$-contractibility of iterated Koras-Russell $3$-folds after simplicial suspension}
In this section we prove \aref{thm:iteratedKR-A1Cont} stated in the introduction.

\begin{theorem}
If $\mathcal{Y}(m,n_{i},\alpha_{i},p)$ is an iterated Koras-Russell $3$-fold of the first kind \eqref{equation:iteratedfirstkindpolynomial}, 
then: 
$$
\Sigma_{s} \mathcal{Y}(m,n_{i},\alpha_{i},p)
\sim_{\AA^{1}}
\ast.
$$ 
\end{theorem}
\begin{proof}
Setting $\AA^{1}_{y}:=\{x=z=t=0\}\subset \mathcal{Y}(m,n_{i},\alpha_{i},p)$ gives us the connecting map 
\begin{equation}
\label{equation:iteratedconnectingmap}
\mathcal{Y}(m,n_{i},\alpha_{i},p)/(\mathcal{Y}(m,n_{i},\alpha_{i},p)\setminus \AA^{1}_{y})\sim_{\AA^{1}}(\mathbb{P}^{1})^{\wedge2}
\stackrel{\partial}{\to}
\Sigma_{s} \mathcal{Y}(m,n_{i},\alpha_{i},p)\setminus \AA^{1}_{y}
\end{equation}
associated to the cofiber sequence: 
\[
\mathcal{Y}(m,n_{i},\alpha_{i},p)\setminus \AA^{1}_{y}\to \mathcal{Y}(m,n_{i},\alpha_{i},p)\to \mathcal{Y}(m,n_{i},\alpha_{i},p)/(\mathcal{Y}(m,n_{i},\alpha_{i},p)\setminus \AA^{1}_{y}).
\]
To conclude $\Sigma_{s} \mathcal{Y}(m,n_{i},\alpha_{i},p)\sim_{\AA^{1}}\ast$ it suffices to show $\partial$ in \eqref{equation:iteratedconnectingmap} is an $\AA^{1}$-weak equivalence. 
First we observe that $\mathcal{Y}(m,n_{i},\alpha_{i},p)\wedge(\mathbb{P}^{1})^{\wedge n}\sim_{\AA^{1}}\ast$ for $n\gg 0$ by \aref{corollary:iterated stable}.
Following the proof of \aref{corollary:ThomSpacesIso} we show $\mathcal{Y}(m,n_{i},\alpha_{i},p)\setminus \AA^{1}_{y}$ is $\AA^{1}$-weakly equivalent to $\AA^{2}\setminus\{(0,0)\}$.
 
Set $\pi:=\mathrm{pr}_{x,t}:\mathcal{Y}(m,n_{i},\alpha_{i},p)\to\AA^{2}$ and define $\varphi:\AA^{2}=\mathrm{Spec}(F[u,v])\to\AA^{2}$ by 
$(u,v)\mapsto(x,t)=(-u^{\alpha_{1} m}-v^{\alpha_2},v)$;
this is a cyclic Galois cover of order $m\alpha_{1}$ which is fully ramified over $\{t=0\}\subset\AA^{2}$ and \'etale elsewhere.
The fiber product $\mathcal{Y}(m,n_{i},\alpha_{i},p)\times_{\AA^{2}_{x,t}}\AA^{2}_{u,v}$ is isomorphic to the closed subscheme 
\[
\widetilde{\mathcal{Y}(m,n_{i},\alpha_{i},p)}
:=
\{(-u^{\alpha_{1} m}-v^{\alpha_2})^{n_{1}}z=((-u^{\alpha_{1} m}-v^{\alpha_2})^{n_2}y+z^{m})^{\alpha_{1}}+v^{\alpha_2}+(-u^{\alpha_1m}-v^{\alpha_2})p(-u^{\alpha_1m}-v^{\alpha_2},y,v)\},
\]
of $\AA^{4}_{u,v,y,z}$.
The Galois group $\mu_{m\alpha_{1}}$ of the cover $\varphi$ of $\widetilde{\mathcal{Y}(m,n_{i},\alpha_{i},p)}$ is generated by $(u,v,y,z)\mapsto(\varepsilon u,v,y,z)$,
where $\varepsilon\in F^{*}$ is a primitive $\alpha_{1}m$-th root of unity. 
Moreover, 
the projection map 
\[
\widetilde{\pi}:=
\mathrm{pr}_{u,v}:
\widetilde{\mathcal{Y}(m,n_{i},\alpha_{i},p)}
\to
\AA^{2}
\]
restricts to a trivial $\AA^{1}$-bundle over the complement in $\AA^{2}$ of the irreducible cuspidal curve $\Gamma_{\alpha_{1}m,\alpha_{2}}=\{u^{\alpha_{1} m}+v^{\alpha_2}=0\}$. 
The fiber of $\widetilde{\pi}$ over the origin $\{(0,0)\}$ is supported by the line $\widetilde{\AA^{1}_{z}}:=\{u=v=z=0\}$,
and $\widetilde{\pi}^{-1}(\Gamma_{\alpha_{1}m,\alpha_{2}})$ is isomorphic to: 
\[
\mathrm{Spec}(F[u,v][y,z]/(-u^{\alpha_{1} m}-v^{\alpha_2},z^{m\alpha_{1}}-u^{m\alpha_{1}})).
\]
The normalization map $\nu:\widetilde{\Gamma_{\alpha_{1}m,\alpha_{2}}}=\AA^{1}_{w}\to\Gamma_{\alpha_{1}m,\alpha_{2}}$,
$w\mapsto(w^{\alpha_2},w^{m\alpha_{1}})$, 
restricts to an isomorphism between $\widetilde{\Gamma_{\alpha_{1}m,\alpha_{2}}}\setminus\{(0,0)\}=\AA_{*}^{1}$ and $\Gamma_{\alpha_{1}m,\alpha_{2}}\setminus\{(0,0)\}$, 
while $\widetilde{\pi}^{-1}(\Gamma_{\alpha_{1}m,\alpha_{2}}\setminus\{(0,0)\})$ splits into a disjoint union of $\alpha_{1}m$ irreducible components 
\[
\mathcal{D}_{\varepsilon^{i}}:\{u^{m\alpha_{1}}-v^{\alpha_2}=z-\varepsilon^{i}u=0\}\setminus\{u=v=z=0\},
0\leq i\leq \alpha_{1}m-1,
\] 
all isomorphic to $\AA_{*}^{1}\times\mathrm{Spec}(F[y])$.
As in the proof of \aref{lem:DefKR-Punctured-Weak-Equiv} we conclude the map 
\[
\widetilde{\pi}:
\widetilde{\mathcal{Y}(m,n_{i},\alpha_{i},p)}\setminus\widetilde{\AA^{1}_{z}}
\to\AA^{2}\setminus\{(0,0)\}
\]
factors through a locally trivial $\AA^{1}$-bundle 
\[
\widetilde{\rho}:
\widetilde{\mathcal{Y}(m,n_{i},\alpha_{i},p)}\setminus\widetilde{\AA^{1}_{z}}
\to
\widetilde{\mathcal{S}}
\]
over a smooth scheme $\widetilde{\delta}:\widetilde{\mathcal{S}}\to\AA^{2}\setminus\{(0,0)\}$.
This is achieved by gluing together the $m\alpha_{1}$-copies $\widetilde{\mathcal{S}}_{\varepsilon^{i}}=\Spec(F[u_{\varepsilon^{i}},v])\setminus\{(0,0)\}$ of $\AA^{2}\setminus\{(0,0)\}$, 
$0\leq i\leq m\alpha_{1}-1$, 
using the identity map away from the curves $\{u_{\varepsilon^{i}}^{m\alpha_{1}}-v^{\alpha_2}=0\}$.
Furthermore, 
$\widetilde{\rho}$ descends to an \'etale locally trivial $\AA^{1}$-bundle 
\[
\rho:\mathcal{Y}(m,n_{i},\alpha_{i},p)\setminus \AA^{1}_{y}\to\mathfrak{\mathcal{S}}(m\alpha_{1},\alpha_2)
\]
over the algebraic space constructed in \aref{lem:DefKR-Punctured-Weak-Equiv}. 
Since $\mathfrak{S}(m\alpha_{1},\alpha_2)=\mathfrak{S}(m\alpha_{1},1)$ there exists a scheme and a total space of Zariski locally trivial $\AA^{1}$-bundles over 
$\mathcal{Y}(m,n_{i},\alpha_{i},p)\setminus \AA^{1}_{y}$ and $\AA^{2}\setminus\{(0,0)\}\times\AA^{1}$. 
This construction furnishes the desired $\AA^{1}$-weak equivalence between $\mathcal{Y}(m,n_{i},\alpha_{i},p)\setminus \AA^{1}_{y}$ and the punctured affine plane $\AA^{2}\setminus\{(0,0)\}$. 
\end{proof}

\subsection*{Acknowledgements}
The authors acknowledge hospitality and support from Institut Mittag-Leffler during Spring 2017, 
and RCN Frontier Research Group Project no.~250399 "Motivic Hopf Equations".
This work received support from  the French "Investissements d'Avenir" program, project ISITE-BFC (contract ANR-lS-IDEX-OOOB). 
{\O}stv{\ae}r is supported by a Friedrich Wilhelm Bessel Research Award from the Humboldt Foundation and a Nelder Visiting Fellowship from Imperial College London.

\bibliographystyle{plain}
\bibliography{DPO}

\end{document}